\documentclass [a4paper,12pt]{article}
\usepackage{amsmath}
\usepackage{amsfonts}
\usepackage{indentfirst}
\usepackage{esint}
\usepackage{latexsym}
\usepackage{bbm}
\usepackage{amsfonts}
\usepackage{mathrsfs}
\usepackage{amssymb}
\usepackage{amsmath}
\usepackage{amsthm}
\usepackage{latexsym}
\usepackage{indentfirst}
\usepackage{enumitem}
\usepackage{bm}
\usepackage{esint}
\usepackage{cite}
\numberwithin{equation}{section}
\usepackage{color}
\allowdisplaybreaks
\usepackage{amsthm}
\usepackage{amssymb}
\usepackage{enumerate}
\usepackage{ulem}
\allowdisplaybreaks

\newtheorem{theorem}{Theorem}[section]
\newtheorem{lemma}[theorem]{Lemma}
\newtheorem{thm}[theorem]{Theorem}

\newtheorem{rmk}[theorem]{Remark}

\makeatletter
\usepackage{amsmath}
\usepackage{amsfonts}
\usepackage{indentfirst}
\usepackage{esint}
\usepackage{latexsym}
\usepackage{color}
\UseRawInputEncoding
\usepackage{amsthm}
\usepackage{amssymb}
\usepackage{enumerate}
\usepackage{ulem}
\makeatletter

\newcommand{\Rmnum}[1]{\expandafter\@slowromancap\romannumeral #1@}
\makeatother
\begin{document}
\title{Estimates of eigenvalues and eigenfunctions in elliptic homogenization with rapidly oscillating potentials}
\author{Yiping Zhang\footnote{Email:zhangyiping161@mails.ucas.ac.cn}\\\small{Academy of Mathematics and Systems Science, CAS;}\\
\small{University of Chinese Academy of Sciences;}\\
\small{Beijing 100190, P.R. China.}}
\date{}
\maketitle
\begin{abstract}
In this paper, for a family of second-order elliptic equations with rapidly oscillating periodic coefficients and rapidly oscillating periodic potentials, we are interested in the $H^1$ convergence rates and the Dirichlet eigenvalues and bounds of the normal derivatives of Dirichlet eigenfunctions. The $H^1$ convergence rates rely on the Dirichlet correctors and the first-order corrector for the oscillating potentials. And  the bound results rely on
an $O(\varepsilon)$ estimate in $H^1$ for solutions with Dirichlet condition.
\end{abstract}

\section{Introduction}
This paper concerns the $H^1$ convergence rates and the asymptotic behavior of Dirichlet eigenvalues and eigenfunctions for a family of second-order elliptic equations with rapidly oscillating coefficients and rapidly oscillating potentials arising from A. Bensoussan, J.-L. Lions and G. Papanicolaou in \cite{1978Asymptotic}. More precisely, consider
\begin{equation}
\left\{
\begin{aligned}
\mathcal{L}_{\varepsilon} u_{\varepsilon} \equiv-\operatorname{div}\left({A}(x/\varepsilon)\nabla u_\varepsilon\right)+\frac{1}{\varepsilon}W\left(x/\varepsilon\right) u_\varepsilon&=f  \text { in } \Omega, \\
u_{\varepsilon}&=0  \text { on } \partial \Omega.
\end{aligned}\right.\ \varepsilon>0,
\end{equation}where $\Omega$ is a bounded Lipschitz domain (the summation convention is used throughout the paper), and $W\in L^\infty(Y)$ is 1-periodic with $\int_Y W(y)dy=0$, where $Y=[0,1)^d\simeq \mathbb{R}^d/\mathbb{Z}^d$ for $d\geq 2$. We will always assume that $A(y)=(a_{ij}(y))$ with $1\leq i, j\leq d$ is real, symmetric ($A^*=A$), and satisfies the ellipticity condition
\begin{equation}
\kappa |\xi|^2\leq a_{ij}(y)\xi_i\xi_j \leq \kappa^{-1}|\xi|^2 \text{\quad for }y\in \mathbb{R}^d \text{ and }\xi\in \mathbb{R}^d,
\end{equation}where $\kappa\in (0,1)$, and the 1-periodicity condition
\begin{equation}
A(y+z)=A(z) \text{\quad for }y\in \mathbb{R}^d\text{ and }z\in \mathbb{Z}^d,
\end{equation}
as well as  the $VMO(\mathbb{R}^d)$ smoothness condition. In order to quantify the smoothness, we will impose the following condition:
\begin{equation}
\sup_{x\in \mathbb{R}^d}\fint_{B(x,t)}|A-\fint_{B(x,t)}A|\leq \rho(t), \text{\quad for }0<t\leq 1,
\end{equation}
where $\rho$ is a nondecreasing continuous function on $[0,1]$ with $\rho(0)=0$.

It is more or less well-known that the effective equation is given by
\begin{equation}\left\{
\begin{aligned}
\mathcal{L}_{0} u_0 \equiv-\operatorname{div}\left(\widehat{A}\nabla u_0\right)+\mathcal{M}(W\chi_w)u_0 &=f  \text { in } \Omega, \\
u_{0}&=0  \text { on } \partial \Omega,
\end{aligned}\right.
\end{equation}
under some suitable conditions (see section 2). The operator $\widehat{A}=(\widehat{a}_{ij}):\Omega\times\mathbb{R}^n\mapsto\mathbb{R}^n$ is a constant matrix defined as

\begin{equation}\widehat{a}_{i j}=\fint_{Y}\left[a_{i j}+a_{i k} \frac{\partial}{\partial y_{k}}\left(\chi_{j}\right)\right](y) d y,\end{equation}
where $\chi_{k}(y)$ is the unique solution of the cell-problem
\begin{equation}\left\{
\begin{aligned}
-\frac{\partial}{\partial y_{i}}\left(a_{i j}\left(y\right) \frac{\partial}{\partial y_{j}}\left(\chi_{k}(y)+y_k\right)\right)=0  \text { in } Y \\
\chi_{k}\  is\  \text{1-periodic},\ \fint_Y\chi_{k}(y)dy=0,
\end{aligned}\right.
\end{equation}
for $k=1,2,\cdots,d$. It is well-known that the coefficient matrix $(\widehat{a}_{ij})$ satisfies the uniformly elliptic condition $\kappa |\xi|^2\leq \widehat{a}_{ij}\xi_i\xi_j \leq \kappa_1|\xi|^2$, where $\kappa_1$ depends only on $\kappa$ and $d$ (for example, see \cite{1978Asymptotic}). And $$\mathcal{M}(W\chi_w)=:\fint_Y (W\chi_w)(y)dy,$$
where $\chi_w$ is the unique solution of the cell-problem
\begin{equation}\left\{
\begin{aligned}
\frac{\partial}{\partial y_{i}}\left(a_{i j}\left(y\right) \frac{\partial}{\partial y_{j}}\chi_w(y)\right)=W(y) \text { in } Y, \\
\chi_w\  is\  \text{1-periodic},\ \fint_Y\chi_{w}(y)dy=0.
\end{aligned}\right.
\end{equation}
The following theorem states that $\mathcal{L}_0$ is the homogenized operator for $\mathcal{L}_\varepsilon$, as $\varepsilon\rightarrow 0$.
\begin{thm}
Assume that $(1.2)$, $(1.3)$ and $(1.4)$ hold true, $f\in L^2(\Omega)$ and $\mathcal{M}(W\chi_w)>-\lambda_{0,1}'$ with $\lambda_{0,1}'$ being
the first eigenvalue of $\mathcal{L}_0'=-\operatorname{div}(\widehat{A}(\nabla\cdot))$ in $\Omega$ for the Dirichlet's boundary condition, where $\Omega$
is a bounded Lipschitz domain, then there exists $\varepsilon_1\in (0,1)$, depending only on $A$, $W$, $d$ and $\Omega$, such that the equation $(1.1)$ admits a
unique solution $u_\varepsilon$ for $0<\varepsilon\leq\varepsilon_1$, and
$$u_\varepsilon\rightharpoonup u_0\quad \text{weakly in }\  H^1_0(\Omega), \text{ as }\varepsilon\rightarrow 0,$$
where $u_0\in H^2(\Omega)\cap H^1_0(\Omega)$ is the solution of $(1.5)$ .
\end{thm}This result is more or less known, and for completeness,
we provide a new proof of Theorem 1.1 in section 3 with some interesting observations after obtaining the $H^1$ convergence rates.
\begin{rmk}
Generally, we can't expect the term $\frac{1}{\varepsilon}W^\varepsilon u_\varepsilon$ is uniformly bounded in $L^2(\Omega)$ even in the case when
$a_{ij}=\delta_{ij}$. If so, then according to the equation $-\Delta u_\varepsilon +\frac{1}{\varepsilon}W^\varepsilon u_\varepsilon=f\in L^2(\Omega)$, we have
that $u_\varepsilon$ is uniformly bounded in $H^2(\Omega)$, for $1\geq\varepsilon>0$, which implies that $u_\varepsilon$ is compact in $H^1(\Omega)$. Then, according to
$\frac{1}{\varepsilon}W^\varepsilon\rightharpoonup 0 $ weakly in $H^{-1}(\Omega)$, we have
$\langle\frac{1}{\varepsilon}W^\varepsilon,u_\varepsilon\rangle\rightarrow 0$, which is often not the case. Actually, in view of $(3.6)$, there holds
$\frac{1}{\varepsilon}W^\varepsilon u_\varepsilon=\varepsilon\Delta_x \psi_{3}^\varepsilon u_\varepsilon=\nabla_x(\nabla_y\psi_{3}^\varepsilon
\cdot u_\varepsilon)-\nabla_y\psi_{3}^\varepsilon\cdot\nabla_x u_\varepsilon$, which states that $\frac{1}{\varepsilon}W^\varepsilon u_\varepsilon$ is uniformly bounded in $H^{-1}(\Omega)$, hence converges weakly in $H^{-1}(\Omega)$ as $\varepsilon\rightarrow 0$.
\end{rmk}

The next theorem states the $H^1$ convergence rates for $0<\varepsilon\leq \varepsilon_1$. We first introduce the following so-called Dirichlet correctors,
defined by
\begin{equation}\left\{\begin{aligned}
-\operatorname{div}(A(x/\varepsilon)\nabla \Phi_{\varepsilon,j}(x))&=0&\text{ \quad in }&\Omega,\\
\Phi_{\varepsilon,j}&=x_j&\text{ \quad on }&\partial \Omega,
\end{aligned}\right.\end{equation} for $1\leq j\leq d$.
\begin{thm}
Under the conditions in Theorem 1.1. We additionally assume that $\Omega$ is a bounded $C^1$ domain or bounded convex Lipschitz domain. For $\varepsilon_1\geq\varepsilon\geq 0$, with $\varepsilon_1$ given in Theorem 1.1, and
$f\in L^2(\Omega)$. Let $u_\varepsilon$ be the unique weak solution in $H^1_0(\Omega)$ to $\mathcal{L}_\varepsilon u_\varepsilon=f$ in $\Omega$. Then
\begin{equation}
\left\|u_\varepsilon-u_0-\left\{\Phi_{\varepsilon,j}(x)-x_j\right\}\partial_{j}u_0-\varepsilon\chi_w^\varepsilon u_0\right\|_{H^1_0(\Omega)}\leq C\varepsilon ||f||_{L^2(\Omega)},
\end{equation}
with $\chi_w^\varepsilon=\chi_w(x/\varepsilon)$, where $C$ depends only on $A$, $W$, $d$ and $\Omega$.
\end{thm}
After obtaining the $H^1$ convergence rates, we could obtain the following asymptotic behavior of spectra of $\mathcal{L}_\varepsilon$, with
$0<\varepsilon\leq \varepsilon_1$. Let $\{\lambda_{\varepsilon,k}\}$ denote the sequence of Dirichlet eigenvalues in an increasing order for
$\mathcal{L}_\varepsilon$ in a bounded Lipschitz domain $\Omega$. We also use $\{\lambda_{0,k}\}$ to denote the sequence of Dirichlet eigenvalues in an
increasing order for $\mathcal{L}_0$ in $\Omega$. Then, there hold the following eigenvalues convergence rates.
\begin{thm}
Suppose that $A$  satisfies conditions $(1.2)$, $(1.3)$ and $(1.4)$ as well as $\mathcal{M}(W\chi_w)>-\lambda_{0,1}'$. Let $\Omega$ be a bounded $C^1$ domain or bounded convex Lipschitz domain in  $\mathbb{R}^d$, $d\geq 2$. Then
\begin{equation}
|\lambda_{\varepsilon,k}-\lambda_{0,k}|\leq C \varepsilon (\lambda_{\varepsilon,k})^{3/2},
\end{equation}for $0< \varepsilon \leq \varepsilon_1$ with $\varepsilon_1$ given in Theorem 1.1, where $C$ is independent of $\varepsilon$ and $k$.
\end{thm}
Asymptotic behavior of spectra of $\mathcal{L}_\varepsilon$ is an important problem in periodic homogenization, and
there are a large of literatures concerning the convergence rates of eigenvalues of the operators
$\mathcal{L}_\varepsilon'=-\operatorname{div}(A(x/\varepsilon)\nabla \cdot)$, we refer readers to
\cite{Kesavan1979Homogenization1,Kesavan1979Homogenization2,santosa1993first,jikov2012homogenization,Moskow1997First,Castro2000High,Castro2000LOW,2012Convergence} for the results about the convergence rates of
eigenvalues and their references therein for more results. In particular, if we denote $\{\lambda_{\varepsilon,k}'\}$ the sequence of Dirichlet
eigenvalues in an increasing order for $\mathcal{L}_\varepsilon'$ and $\{\lambda_{0,k}'\}$ for $\mathcal{L}_0'=-\operatorname{div}(\widehat{A}\nabla
\cdot)$, respectively. Then the estimate
\begin{equation}|\lambda_{\varepsilon,k}'-\lambda_{0,k}'|\leq C \varepsilon (\lambda_{0,k}')^2\end{equation}
may be found in \cite{2012Convergence} for smooth domains, with the coefficient matrix $A$ satisfying $(1.2)$ and $(1.3)$, where $C$ is independent of $\varepsilon$ and $k$. In 2013, Carlos E. Kenig, F. Lin and Z. Shen \cite{kenig2013estimates} improve the
estimates by a factor of $(\lambda_{0,k}')^{1/2}$, which is achieved by utilizing the following $O(\varepsilon)$ estimate in $H^1_0(\Omega)$:
\begin{equation}
\left\|u_\varepsilon-u_0-\left\{\Phi_{\varepsilon,j}(x)-x_j\right\}\partial_{j}u_0\right\|_{H^1_0(\Omega)}\leq C\varepsilon ||u_0||_{H^2(\Omega)},
\end{equation}
under the assumption that $A$ satisfies $(1.2)$, $(1.3)$ and $\Omega$ is a convex bounded Lipschitz domain. The proof of Theorem 1.4 follows the ideas in
\cite{kenig2013estimates}, after obtaining the $O(\varepsilon)$ estimate in $H^1_0(\Omega)$, parallel to $(1.13)$. Note that the elliptic systems
in homogenization have also been studied in \cite{kenig2013estimates}, which we omit it here.

In the rest of the paper, we pay attention to the bounds of the conormal derivatives of Dirichlet eigenfunctions for $\mathcal{L}_\varepsilon$. There are a large of literatures concerning the bounds of the conormal derivatives of Dirichlet eigenfunctions for $\mathcal{L}_\varepsilon'$. We refer readers to
\cite{castro1999boundary,Castro2000High,Castro2000LOW,ozawa1993asymptotic,xu2012upper}
and their references therein for more results. We  especially point out that the following two theorem have been obtained in \cite{kenig2013estimates} for  $\mathcal{L}_\varepsilon'$, and we could extend these results to $\mathcal{L}_\varepsilon$ with rapidly oscillating periodic potentials.
\begin{thm}
Suppose that $A$ satisfies conditions $(1.2)$ and $(1.3)$. Also assume that $A$ is Lipschitz continuous. Let $\Omega$ be a bounded $C^{1,1}$ domain in $\mathbb{R}^d$, $d\geq 2$. Let $\varphi_\varepsilon\in H^1_0(\Omega)$ be a Dirichlet eigenfunction for $\mathcal{L}_\varepsilon$ in $\Omega$ with the associated eigenvalue $\lambda_\varepsilon$ and  $||\varphi_\varepsilon||_{L^2(\Omega)}=1$. Then

\begin{equation}\int_{\partial \Omega}\left|\nabla \varphi_{\varepsilon}\right|^{2} d \sigma \leq\left\{\begin{array}{ll}
C \lambda_\varepsilon\left(1+\varepsilon^{-1}\right) & \text {if } \varepsilon^{2} \lambda_\varepsilon \geq 1, \\
C \lambda_\varepsilon(1+\varepsilon \lambda_\varepsilon) & \text { if } \varepsilon^{2} \lambda_\varepsilon<1,
\end{array}\right.\end{equation}
 where $C$ depends only on $A$, $W$, $d$ and $\Omega$.
\end{thm}
If $\varepsilon \lambda_\varepsilon$ is sufficiently small, we also obtain a sharp lower bound.
\begin{thm}
Let $\Omega$ be a bounded $C^2$ domain in $\mathbb{R}^d$, $d\geq 2$. Suppose that $A$ satisfies the conditions in Theorem 1.5. Let $\varphi_\varepsilon\in H^1_0(\Omega)$ be a Dirichlet eigenfunction for $\mathcal{L}_\varepsilon$ in $\Omega$ with the associated eigenvalue $\lambda_\varepsilon$ and  $||\varphi_\varepsilon||_{L^2(\Omega)}=1$. Then there exists $\delta>0$ such if $\lambda_\varepsilon>1$ and $\varepsilon \lambda_\varepsilon< \delta$, then
\begin{equation}
\int_{\partial\Omega}|\nabla\varphi_\varepsilon|^2 d\sigma \geq c\lambda_\varepsilon,
\end{equation}
 where $\delta>0$ and $c>0$ depend only on $A$, $W$, $d$ and $\Omega$.
\end{thm}

For $0<\varepsilon\leq \varepsilon_1$, let $\{\varphi_{\varepsilon,k}\}$ be an orthonormal basis of $L^2(\Omega)$, where ${\varphi_{\varepsilon,k}}$ is a Dirichlet eigenfunction for $\mathcal{L}_\varepsilon$ in $\Omega$ with eigenvalue $\lambda_{\varepsilon,k}$. The spectral (cluster) projection operator $S_{\varepsilon,k}(f)$ is defined by
\begin{equation}
S_{\varepsilon,\lambda}(f)=\sum_{\sqrt{\lambda_{\varepsilon,k}}\in [\sqrt{\lambda},\sqrt{\lambda+1})}\varphi_{\varepsilon,k}(f),
\end{equation}
where $\lambda\geq 1$, $\varphi_{\varepsilon,k}(f)(x)=\langle\varphi_{\varepsilon,k},f\rangle\varphi_{\varepsilon,k}(x)$, and $\langle,\rangle$ denotes the inner product in
$L^2(\Omega)$. Let $u_\varepsilon=S_{\varepsilon,k}(f)$, where $f\in L^2(\Omega)$ and $||f||_{L^2(\Omega)}=1$. We will show in section 5 that

\begin{equation}\int_{\partial \Omega}\left|\nabla u_{\varepsilon}\right|^{2} d \sigma \leq\left\{\begin{array}{ll}
C \lambda\left(1+\varepsilon^{-1}\right) & \text {if } \varepsilon^{2} \lambda \geq 1, \\
C \lambda(1+\varepsilon \lambda) & \text { if } \varepsilon^{2} \lambda<1,
\end{array}\right.\end{equation}
where $C$ depends only on $A$, $W$, $d$ and $\Omega$. Theorem 1.5 follows if we choose $f$ to be an eigenfunction of
$\mathcal{L}_\varepsilon$. We point out that the estimate in $(1.17)$ for the case $\varepsilon^2 \lambda\geq 1$, as in the case of Laplacian
\cite{xu2012upper}, follows readily from the Rellich identities, the proof for the case $\varepsilon^2 \lambda< 1$ is more subtle. The basic idea is to
use the $H^1$ convergence estimate to approximate the eigenfunction $\varphi_{\varepsilon,k}$ with $\lambda_{\varepsilon,k}$ by the solution
$v_\varepsilon$ of the Dirichlet problem $\mathcal{L}_0(v_\varepsilon)=\lambda_\varepsilon \varphi_\varepsilon$ in $\Omega$ with $v_\varepsilon=0$ on
$\partial\Omega$. Note that the $H^2$ estimate of $\varphi_\varepsilon$ in  $\Omega_{c\varepsilon}=\{x\in \Omega :\text{dist}(x,\partial\Omega)<c\varepsilon\}$ will be needed in the proof of Theorem 1.5. The same approach, together with a compactness argument, also
leads to the lower bounds in Theorem 1.6, whose proof is given in section 6.
\begin{rmk}
Note that we always assume that $0<\varepsilon \leq \varepsilon_1$, since we may obtain the coercive estimates $\langle \mathcal{L}_{\varepsilon}(u),u\rangle\geq c
||u||_{H_0^1(\Omega)}^2$,  for any $u\in H_0^1(\Omega)$, under the assumptions that $0<\varepsilon \leq \varepsilon_1$, $A\in \text{VMO}(\mathbb{R}^d)$ and
$\mathcal{M}(W\chi_w)>-\lambda_{0,1}'$. Moreover, there is an another way to obtain the coercive estimates by considering
$\mathcal{L}_{\varepsilon,\mu}(u_\varepsilon)=-\operatorname{div}(A^\varepsilon \nabla u_\varepsilon)+\frac{1}{\varepsilon}W^\varepsilon u_\varepsilon+\mu u_\varepsilon$ for $\mu>0$ large enough, with $0<\varepsilon\leq 1$. See section 2 for more details.
\end{rmk}
Throughout this paper, we use the following notation

$$H^m_{\text{per}}(Y)=:\left\{f\in H^m(Y) \text{ and }f\text{ is 1-periodic with }\fint_Yfdy=0\right\}, $$
and we will write $\partial_{x_i}$ as $\partial_i$, $F^\varepsilon=F(x/\varepsilon)$ for a function $F$ and $\mathcal{M}(G)=\int_Y G(y)dy$ for a 1-periodic function $G$ if the context is understand.

\section{Estimates for the spectrum and homogenization}
In order to move forward, we first need to know in what condition, that the equation $(1.1)$, with $f\in L^2(\Omega)$, admits a unique solution
$u_\varepsilon\in H^1_0(\Omega)$. If we multiply $(1.1)$ by $u_\varepsilon$, we find
\begin{equation}\int_\Omega A^\varepsilon \nabla u_\varepsilon \nabla u_\varepsilon dx+\frac{1}{\varepsilon}\int_\Omega W^\varepsilon u^2_\varepsilon
dx=\langle f,u_\varepsilon\rangle,
\end{equation}
where $A^\varepsilon=A(x/\varepsilon)$ and$\langle,\rangle$ denotes the inner product in $L^2(\Omega)$.
 The expression $(2.1)$ does not give us any simple a prior estimate, and it is even unclear at this stage if we could solve for all $\varepsilon\rightarrow0$. Actually, there are two ways to investigate the existence of the solution to the equation $(1.1)$. The first way is to consider, for $\mu>0$,
\begin{equation}
\mathcal{L}_{\varepsilon,\mu}(u_\varepsilon)=-\operatorname{div}(A^\varepsilon \nabla u_\varepsilon)+\frac{1}{\varepsilon}W^\varepsilon u_\varepsilon+\mu u_\varepsilon.
\end{equation}
Since $W\in L^\infty(Y)$ is 1-periodic with $\int_Y W(y)dy=0$, then there exists
$\psi_{3}\in H^2_{\text{per}}(Y)\cap W^{2,p}(Y)$, with any $p\in(1,\infty)$, solving
\begin{equation}
\Delta_y \psi_{3}(y)=W(y) \text{ in Y, with }\int_Y \psi_{3}(y)dy=0.
\end{equation}
Thus, for any $u\in H^1_0(\Omega)$, we have
\begin{equation}\begin{aligned}
\frac{1}{\varepsilon}\int_\Omega W^\varepsilon u^2dx&=\varepsilon\int_\Omega \Delta_x \psi_{3}^\varepsilon u^2dx
=-2\int_\Omega\nabla_y \psi_{3}^\varepsilon \nabla_x u\cdot u\\
&\geq -\delta ||\nabla u||_{L^2(\Omega)}^2-C_\delta || u||_{L^2(\Omega)}^2,
\end{aligned}\end{equation}
for any $\delta\in (0,1)$ and $\psi_3^\varepsilon=\psi_3(x/\varepsilon)$. Consequently, if $\mu>0$ is suitably large, then
\begin{equation}
\langle\mathcal{L}_{\varepsilon,\mu}(u),u\rangle\geq C ||u||_{H_0^1(\Omega)}^2, \text{\quad for any }u\in H_0^1(\Omega),
\end{equation}
 which states that the equation $(1.1)$ admits a unique solution in $H_0^1(\Omega)$.

 The second way is to estimate the first eigenvalue
 $\lambda_{1,\varepsilon}$ for the
operator $\mathcal{L}_\varepsilon$, and eventually, we will obtain the coercive estimates for the operator $\mathcal{L}_\varepsilon$, under the assumption  that $\mathcal{M}(W\chi_w)>-\lambda_{0,1}'$, where $\lambda_{0,1}'$ is
the first eigenvalue of $\mathcal{L}_0'=-\operatorname{div}(\widehat{A}(\nabla\cdot))$ in $\Omega$ for the Dirichlet's boundary condition as well as $A$ satisfies the adtional smoothness condition $(1.4)$. However, in view of $(1.8)$, then integration by parts yields
\begin{equation}
\mathcal{M}(W\chi_w)=-\int_YA \nabla \chi_w \nabla\chi_wdy.
\end{equation}
If $\mathcal{M}(W\chi_w)\leq-\lambda_{0,1}'$, then we may consider the operator $\mathcal{L}_{\varepsilon,\mu}$ such that
$\mathcal{M}(W\chi_w)+\mu>-\lambda_{0,1}'$. For simplicity, we may assume $\mu=0$. To obtain the coercive estimates, we need to investigate the asymptotic behavior
for the first eigenvalue for the operator $\mathcal{L}_\varepsilon$ in $\Omega$ for the Dirichlet's boundary condition, which is stated in the following
theorem.

\begin{thm}Assume that the matrix $A(y)$ satisfies the conditions $(1.2)$, $(1.3)$ and $(1.4)$. And let $\lambda_{\varepsilon,1}$, $\lambda_{\varepsilon,1}'$ and $\lambda_{0,1}'$ be the first eigenvalue of the operator $\mathcal{L}_\varepsilon$, $\mathcal{L}_\varepsilon'=\mathcal{L}_\varepsilon-\frac{1}{\varepsilon}W^\varepsilon$ and
$\mathcal{L}_0'=\mathcal{L}_0-\mathcal{M}(W\chi_w)$ in $\Omega$ for the Dirichlet boundary condition, respectively. Then, there exists $\varepsilon_{0}\in (0,1)$, depending only on $A$, $W$, $d$ and  $\Omega$, such that there holds the following
convergence rate for the first eigenvalue,
\begin{equation}
\left|\lambda_{\varepsilon,1}-(\lambda_{\varepsilon,1}'+\mathcal{M}(W\chi_w))\right|\leq C \varepsilon\left(1+\lambda_{\varepsilon,1}'+|\mathcal{M}(W\chi_w)|\right)
\end{equation} and
\begin{equation}
\left|\lambda_{\varepsilon,1}-(\lambda_{0,1}'+\mathcal{M}(W\chi_w))\right|\leq C \varepsilon,
\end{equation} for $0<\varepsilon \leq \varepsilon_0$, where $C$ depends only on $A$, $W$, $d$ and $\Omega$.
\end{thm}
\begin{proof}We first note that $(2.8)$ directly follows from $(2.7)$, since $|\lambda_{0,1}'-\lambda_{\varepsilon ,1}'|\leq C \varepsilon$ \cite{kenig2013estimates}. Thus, we need only to show $(2.7)$.
Let $v_\varepsilon=(1+\varepsilon\chi_w^\varepsilon)\phi_\varepsilon$ with $\phi_\varepsilon\in H^1_0(\Omega)$, then direct computation shows that
\begin{equation}\begin{aligned}
\mathcal{L}_\varepsilon v_\varepsilon&=-\partial_{x_i}\left[\left(1+\varepsilon\chi_w^\varepsilon\right)a_{ij}^\varepsilon\partial_{x_j}\phi_\varepsilon\right]
-\partial_{x_i}\left[\phi_\varepsilon a_{ij}^\varepsilon\partial_{y_j}\chi_w^\varepsilon\right]+\frac{1}{\varepsilon}W^\varepsilon (1+\varepsilon\chi_w^\varepsilon)\phi_\varepsilon\\
&=-\partial_{x_i}\left[\left(1+\varepsilon\chi_w^\varepsilon\right)a_{ij}^\varepsilon\partial_{x_j}\phi_\varepsilon\right]
-a_{ij}^\varepsilon\partial_{y_j}\chi_w^\varepsilon\partial_{x_i}\phi_\varepsilon+W^\varepsilon\chi_w^\varepsilon\phi_\varepsilon,
\end{aligned}\end{equation} where we have used the equation $(1.8)$ in the second equality in $(2.9)$.

Denote \begin{equation}
\Pi_\varepsilon(v_\varepsilon)=:\left\langle \mathcal{L}_\varepsilon v_\varepsilon, v_\varepsilon\right\rangle.
\end{equation}
 Then according to $(2.9)$, we have
\begin{equation}\begin{aligned}
\Pi_\varepsilon(v_\varepsilon)&=\int_\Omega \left[\left(1+\varepsilon\chi_w^\varepsilon\right)a_{ij}^\varepsilon\partial_{x_j}\phi_\varepsilon\partial_{x_i}v_\varepsilon
-a_{ij}^\varepsilon\partial_{y_j}\chi_w^\varepsilon\partial_{x_i}\phi_\varepsilon v_\varepsilon +W^\varepsilon\chi_w^\varepsilon\phi_\varepsilon v_\varepsilon\right]dx\\
&=\int_\Omega \left(1+\varepsilon\chi_w^\varepsilon\right)^2a_{ij}^\varepsilon\partial_{x_j}\phi_\varepsilon\partial_{x_i}\phi_\varepsilon dx
+\int_\Omega W^\varepsilon\chi_w^\varepsilon\phi_\varepsilon^2\left(1+\varepsilon\chi_w^\varepsilon\right)dx,
\end{aligned}\end{equation} where we have used the symmetry of the coefficients $a_{ij}$ and $v_\varepsilon=(1+\varepsilon\chi_w^\varepsilon)\phi_\varepsilon$ in the above equality. Let $\psi_2=\psi_2(y)\in H^2_{\text{per}}(Y)\cap W^{2,p}(Y)$ with $1<p<\infty$, solve the equation
\begin{equation}
-\Delta_y\psi_2 -W\chi_w+\mathcal{M}(W\chi_w)=0 \text{ in Y with }\int_Y\psi_2(y)dy=0,
\end{equation}since $||\chi_w||_\infty\leq C$ due to the De Giorgi-Nash-Moser theorem.
 Then, there holds

\begin{equation}\begin{aligned}
&\int_\Omega W^\varepsilon\chi_w^\varepsilon\phi_\varepsilon^2\left(1+\varepsilon\chi_w^\varepsilon\right)dx\\
=&\int_\Omega \left[-\varepsilon^2\Delta_x \psi_2^\varepsilon +\mathcal{M}(W\chi_w)\right]\phi_\varepsilon^2\left(1+\varepsilon\chi_w^\varepsilon\right)dx\\
=&\mathcal{M}(W\chi_w)\int_\Omega \phi_\varepsilon^2\left(1+\varepsilon\chi_w^\varepsilon\right)dx+ \varepsilon \int_\Omega \partial_{y_i}\psi_2^\varepsilon\cdot\partial_{x_i}[\phi_\varepsilon^2\left(1+\varepsilon\chi_w^\varepsilon\right)]dx.
\end{aligned}\end{equation}
Simple computation shows that
\begin{equation}\begin{aligned}
\int_\Omega\phi_\varepsilon^2|\nabla_y \chi_w^\varepsilon|dx
&\leq\left\{\begin{aligned}
&||\phi_\varepsilon||_{L^4}^2||\nabla_y \chi_w^\varepsilon||_{L^2}, \text{ if }d=2,3,4\\
&||\phi_\varepsilon||_{L^\frac{2d}{d-2}}^2||\nabla_y \chi_w^\varepsilon||_{L^\frac{d}{2}}, \text{ if }d\geq 5
\end{aligned}\right.\\
&\leq C \left(||\phi_\varepsilon||_{L^2}^2+||\nabla \phi_\varepsilon||_{L^2}^2\right),
\end{aligned}\end{equation} where we have used the Galiardo-Nirenberg inequality for $d=2,3$ and the Sobolev inequality for $d\geq 4$, as well as the following inequality
\begin{equation}\begin{aligned}
\int_\Omega |\nabla_y \chi_w^\varepsilon|^{p}dx&=\varepsilon^d\int_{\tilde{\Omega}} |\nabla_y \chi_w(y)|^{p}dy\\
&\leq C \int_Y |\nabla_y \chi_w(y)|^{p}dy \leq C \int_{B_2(0)} |\nabla_y \chi_w(y)|^{p}dy\\
&\leq C\left(\int_{B_4(0)} |\nabla_y \chi_w(y)|^{2}dy\right)^{p/2}\\
&\leq C_p,
\end{aligned}\end{equation}
with $\tilde{\Omega}=\{y:y=\frac{x}{\varepsilon},x\in \Omega\}$, $2<p<\infty$, and $C_p$ depending only on $A$, $p$ and $\Omega$, where we have used the periodicity of $\chi_w$ in the second and the fourth inequality, and the $W^{1,p}$ estimates for $\chi_w$ under the assumption $A=(a_{ij})\in VMO(\mathbb{R}^d)$ \cite{Shen2004Bounds}. In view of $\nabla_y \psi_2\in L^\infty(\mathbb{R}^d)$ and $(2.14)$, then there holds

\begin{equation}\begin{aligned}
\varepsilon\left|\int_\Omega \partial_{y_i}\psi^\varepsilon
\partial_{x_i}[\phi_\varepsilon^2\left(1+\varepsilon\chi_w^\varepsilon\right)]dx\right|&\leq C \varepsilon\int_\Omega \left[|\phi_\varepsilon|\cdot|\nabla \phi_\varepsilon|+\phi_\varepsilon^2|\nabla_y \chi_w^\varepsilon|\right]dx\\
&\leq C \varepsilon\left(||\phi_\varepsilon||_{L^2(\Omega)}^2+||\nabla \phi_\varepsilon||_{L^2(\Omega)}^2\right),
\end{aligned}\end{equation} where we have used the $||\chi_w||_\infty \leq C$ .
Consequently, combining $(2.12)$-$(2.16)$ yields
\begin{equation}\begin{aligned}
&(1-C\varepsilon)\int_\Omega a_{ij}^\varepsilon\partial_{x_j}\phi_\varepsilon\partial_{x_i}\phi_\varepsilon dx+(1+C\varepsilon)\mathcal{M}(W\chi_w)\int_\Omega \phi_\varepsilon^2dx-C\varepsilon ||\phi_\varepsilon||_{L^2(\Omega)}^2
\leq \Pi(v_\varepsilon)\\
&\leq(1+C\varepsilon)\int_\Omega a_{ij}^\varepsilon\partial_{x_j}\phi_\varepsilon\partial_{x_i}\phi_\varepsilon dx+(1-C\varepsilon)\mathcal{M}(W\chi_w)\int_\Omega \phi_\varepsilon^2dx+C\varepsilon ||\phi_\varepsilon||_{L^2(\Omega)}^2,
\end{aligned}\end{equation}
 In view of $(1-C\varepsilon)||v_\varepsilon||_{L^2}^2\leq ||\phi_\varepsilon||_{L^2}^2\leq (1+C\varepsilon)||v_\varepsilon||_{L^2}^2$ for any $\phi_\varepsilon\in H^1_0(\Omega)$, then
\begin{equation}
\left|\lambda_{\varepsilon,1}-(\lambda_{\varepsilon,1}'+\mathcal{M}(W\chi_w))\right|\leq C \varepsilon\left(1+\lambda_{\varepsilon,1}'+|\mathcal{M}(W\chi_w)|\right)
\end{equation}after choosing $\varepsilon_0$ such that $1-C \varepsilon_0\geq1/4$ with $0<\varepsilon \leq \varepsilon_0$.
Thus, we complete this proof of $(2.7)$.
\end{proof}
Under the conditions in Theorem 2.1, if we additionally assume that $\lambda_{0,1}'+\mathcal{M}(W\chi_w)>0$, then we could choose $\varepsilon_1$ small such that

\begin{equation}\varepsilon_1\leq\text{ min}\left\{\varepsilon_0, \frac{\lambda_{0,1}'+\mathcal{M}(W\chi_w)}{2C}\right\},
\end{equation}
with the same constant $C$ in $(2.8)$,  then
\begin{equation}\lambda_{\varepsilon,1}\geq \frac{\lambda_{0,1}'+\mathcal{M}(W\chi_w)}{2C}>0\end{equation} for $0<\varepsilon\leq \varepsilon_1$, which implies that

\begin{equation}
\langle \mathcal{L}_{\varepsilon}(u),u\rangle\geq c ||u||_{H_0^1(\Omega)}^2, \text{\quad for any }u\in H_0^1(\Omega),
\end{equation}
for some constant $c>0$. Similarly, it is easy to see that, for any $u\in H_0^1(\Omega)$,
\begin{equation}\begin{aligned}
\langle\mathcal{L}_{0}(u),u\rangle&=\int_\Omega \widehat{a}_{ij}\partial_iu\partial_j udx+\mathcal{M}(W\chi_w)\int_\Omega u^2dx\\
&\geq \left(1+\frac{\mathcal{M}(W\chi_w)}{\lambda_{0,1}'}\right)\int_\Omega \widehat{a}_{ij}\partial_iu\partial_judx\\
&\geq c\int_\Omega \widehat{a}_{ij}\partial_iu\partial_judx,
\end{aligned}\end{equation}
where we have used $\mathcal{M}(W\chi_w)<0$ in the second inequality. Consequently, if $u_0\in H^2(\Omega)\cap H^1_0(\Omega)$ satisfies
$\mathcal{L}_0(u_0)=f\in L^2(\Omega)$, it is easy to see that
\begin{equation}
||\nabla u_0||_{L^2(\Omega)}\leq C||f||_{L^2(\Omega)}
\end{equation}
and \begin{equation}
||\nabla^2 u_0||_{L^2(\Omega)}\leq C||f||_{L^2(\Omega)}+C |\mathcal{M}(W\chi_w)|\cdot|| u_0||_{L^2(\Omega)}\leq C||f||_{L^2(\Omega)}.
\end{equation}
\section{Convergence rate in $H^1$}
Denote the so-called flux correctors $b_{ij}$ by
\begin{equation}
b_{ij}(y)=\widehat{a}_{ij}-a_{ij}(y)-a_{ik}(y)\frac{\partial\chi_j(y)}{ \partial{y_k}},
\end{equation}
where $1\leq i,j\leq d$.
\begin{lemma}
Suppose that $A$ satisfies conditions $(1.2)$ and $(1.3)$. For $1\leq i,j,k\leq d$, there exists $F_{ijk}\in H^1_{\text{per}}(Y)\cap L^\infty(Y)$ such that
\begin{equation}
b_{ij}=\frac{\partial}{ \partial_{y_k}}F_{kij}\ \text{ and }\ F_{kij}=-F_{ikj}.
\end{equation}
\begin{proof}See \cite[Remark 2.1]{kenig2014periodic}.\end{proof}
\end{lemma}

\begin{lemma}
Suppose that $A$ satisfies conditions $(1.2)$ and $(1.3)$. Let $\Omega$ be a bounded Lipschitz domain. Then
\begin{equation}||\Phi_{\varepsilon,j}(x)-x_j||_{L^\infty(\Omega)}\leq C \varepsilon,\end{equation}
where $C$ depends only on $A$.
\end{lemma}
\begin{proof}
One may find this proof in \cite[Proposition 2.4]{kenig2014periodic}, and we provide it for completeness. Consider $u_\varepsilon(x)=\Phi_{\varepsilon,j}(x)-x_j-\varepsilon\chi_j(x/\varepsilon)$, then $\mathcal{L}_\varepsilon'(u_\varepsilon)=0$ in $\Omega$ and $u_\varepsilon(x)=-\varepsilon\chi_j(x/\varepsilon)$ on $\partial\Omega$. Then one may use the maximum principle and boundedness of $\chi$ to show that $||u_\varepsilon||_{L^\infty(\Omega)}\leq ||u_\varepsilon||_{L^\infty(\partial\Omega)}\leq C\varepsilon$. This implies that $||\Phi_{\varepsilon,j}(x)-x_j-\varepsilon\chi_j(x/\varepsilon)||_{L^\infty(\Omega)} \leq C\varepsilon$.
\end{proof}

In order to move forward, we give some notations first.
It is easy to see that  integration by parts yields that
\begin{equation}\begin{aligned}
\int_Y a_{ij}\partial_{y_j}\chi_w&=-\int_Y \partial_{y_j}a_{ij}\chi_w=-\int_Y \partial_{y_j}a_{ji}\chi_w\\
&=\int_Y\partial_{y_j}(a_{jk}\partial_{y_k}\chi_i)\chi_w=-\int_Ya_{jk}\partial_{y_k}\chi_i\partial_{y_j}\chi_w\\
&=-\int_Y\partial_{y_k}\chi_ia_{kj}\partial_{y_j}\chi_w=\int_Y\chi_i\partial_{y_k}\left(a_{kj}\partial_{y_j}\chi_w\right)\\
&=\int_Y \chi_iW,
\end{aligned}\end{equation}
where we have used the symmetry of $A$, $(1.7)$ and $(1.8)$.
Then there exist $\psi_{1,i}\in H^2_{\text{per}}(Y)\cap W^{2,p}(Y)$, with any $p\in(1,\infty)$ and $i=1,\cdots,d$, solving the following equation
\begin{equation}
\Delta_y \psi_{1,i}(y)=a_{ij}\partial_{y_j}\chi_w-W\chi_i \text{ in Y with } \int_Y\psi_{1,i}(y)dy=0,
\end{equation} since $\nabla_y\chi_w\in L^p(Y)$ which is due to the $W^{1,p}$ estimates under the assumption $A\in VMO(\mathbb{R}^d)$ \cite{Shen2004Bounds}, and $\chi_i\in L^\infty(Y)$ due to the De Giorgi-Nash-Moser theorem, where we have used the symmetry of $a_{ij}$, $(1.7)$ and $(1.8)$.
Similarly, there exist $\psi_{3}\in H^2_{\text{per}}(Y)\cap W^{2,p}(Y)$, with any $p\in(1,\infty)$, solving
\begin{equation}
\Delta_y \psi_{3}(y)=W \text{ in Y with } \int_Y\psi_{3}(y)dy=0,
\end{equation}since $W\in L^\infty(Y)$ with $\int_Y W(y)dy=0$.

\begin{lemma}
Suppose that $u_\varepsilon\in H^1_0(\Omega)$, $u_0\in H^2(\Omega)\cap H^1_0(\Omega)$, and $\mathcal{L}_\varepsilon=\mathcal{L}_0$ in $\Omega$. Let
\begin{equation}
w_\varepsilon= u_\varepsilon-u_0-\left\{\Phi_{\varepsilon,j}(x)-x_j\right\}\partial_{j}u_0-\varepsilon\chi_w^\varepsilon u_0.
\end{equation}Then
\begin{equation}\begin{aligned}
\mathcal{L}_\varepsilon w_\varepsilon
&=\varepsilon\partial_{i}\left\{F_{kij}^\varepsilon\partial^2_{kj}u_0\right\}
+a_{ij}^\varepsilon\partial_{j}\left\{\Phi_{\varepsilon,k}(x)-x_k-\varepsilon\chi_k^\varepsilon\right\}\partial^2_{ik}u_0
+\partial_{x_i}\left(a_{ij}^\varepsilon\varepsilon\chi_w^\varepsilon\partial_{j} u_0\right)\\
&\quad+\partial_{i}\left\{a_{ij}^\varepsilon \left\{\Phi_{\varepsilon,k}(x)-x_k\right\}\partial^2_{jk}u_0\right\}
-\varepsilon\Delta_x \psi_3^\varepsilon\left\{\Phi_{\varepsilon,j}(x)-x_j-\varepsilon\chi_j^\varepsilon\right\}\partial_{j}u_0\\
&\quad+\varepsilon^2 \Delta_x \psi_{1,i}^\varepsilon \partial_{x_i}u_0
-\varepsilon^2 \Delta_x \psi_{2}^\varepsilon u_0,
\end{aligned}\end{equation}
with $\psi_1$, $\psi_2$ and $\psi_3$ defined in $(3.5)$, $(2.12)$ and $(3.6)$, respectively.
\end{lemma}
\begin{proof}
Let $w_\varepsilon= u_\varepsilon-u_0-\left\{\Phi_{\varepsilon,j}(x)-x_j\right\}\partial_{j}u_0-\varepsilon\chi_w^\varepsilon u_0$, and $w_\varepsilon=w_\varepsilon'-\varepsilon\chi_w^\varepsilon u_0$, then directly computation shows that
\begin{equation}
a_{ij}^\varepsilon \partial_{j}w_\varepsilon'=a_{ij}^\varepsilon \partial_{j}u_\varepsilon-a_{ij}^\varepsilon \partial_{j}u_0-a_{ij}^\varepsilon \partial_{j}\left\{\Phi_{\varepsilon,k}(x)-x_k\right\}\partial_{k}u_0-a_{ij}^\varepsilon \left\{\Phi_{\varepsilon,k}(x)-x_k\right\}\partial^2_{jk}u_0,
\end{equation}
Then, according to $\mathcal{L}_\varepsilon u_\varepsilon=\mathcal{L}_0 u_0$ and $-\partial_i\left\{a_{ij}^\varepsilon\partial_{j}\left(\Phi_{\varepsilon,k}(x)-x_k-\varepsilon\chi_k^\varepsilon\right)\right\}=0$ in $\Omega$, we have
\begin{equation}\begin{aligned}
&-\partial_{i}\left(a_{ij}^\varepsilon \partial_{j}w_\varepsilon'\right)\\
=&-\partial_{i}\left(\widehat{a}_{ij} \partial_{j}u_0\right)
+\mathcal{M}(W\chi_w)u_0-\frac{1}{\varepsilon}W^\varepsilon u_\varepsilon+\partial_{i}\left\{a_{ij}^\varepsilon \partial_{j}u_0\right\}\\
&+\partial_{i}\left\{a_{ij}^\varepsilon \partial_{j}\left\{\Phi_{\varepsilon,k}(x)-x_k\right\}\partial_{k}u_0\right\}
+\partial_{i}\left\{a_{ij}^\varepsilon \left\{\Phi_{\varepsilon,k}(x)-x_k\right\}\partial^2_{jk}u_0\right\}\\
=&-\partial_{i}\left\{[\widehat{a}_{ij}-a_{ij}^\varepsilon] \partial_{j}u_0\right\}+\partial_{i}\left\{a_{ij}^\varepsilon \left\{\Phi_{\varepsilon,k}(x)-x_k\right\}\partial^2_{jk}u_0\right\}+\mathcal{M}(W\chi_w)u_0\\
&+a_{ij}^\varepsilon\partial_{j}\left\{\Phi_{\varepsilon,k}(x)-x_k\right\}\partial^2_{ik}u_0
-\frac{1}{\varepsilon}W^\varepsilon u_\varepsilon+\partial_{i}\left\{a_{ij}^\varepsilon\partial_{j}(\varepsilon\chi_k^\varepsilon)\right\}\partial_{k}u_0\\
=&-\partial_{i}\left\{[\widehat{a}_{ij}-a_{ij}^\varepsilon-a_{ik}^\varepsilon\partial_{y_k}\chi_j^\varepsilon] \partial_{j}u_0\right\}
+a_{ij}^\varepsilon\partial_{j}\left\{\Phi_{\varepsilon,k}(x)-x_k-\varepsilon\chi_k^\varepsilon\right\}\partial^2_{ik}u_0\\
&+\partial_{i}\left\{a_{ij}^\varepsilon \left\{\Phi_{\varepsilon,k}(x)-x_k\right\}\partial^2_{jk}u_0\right\}+\mathcal{M}(W\chi_w)u_0
-\frac{1}{\varepsilon}W^\varepsilon u_\varepsilon,
\end{aligned}\end{equation}
where we have used the follow equality
$$\begin{aligned}&\partial_{i}\left\{a_{ij}^\varepsilon \partial_{j}\left\{\Phi_{\varepsilon,k}(x)-x_k\right\}\partial_{k}u_0\right\}\\
=&\partial_{i}\left\{a_{ij}^\varepsilon \partial_{j}\left\{\Phi_{\varepsilon,k}(x)-x_k\right\}\right\}\partial_{k}u_0+a_{ij}^\varepsilon \partial_{j}\left\{\Phi_{\varepsilon,k}(x)-x_k\right\}\partial^2_{ik}u_0\\
=&\partial_{i}\left\{a_{ij}^\varepsilon\partial_{j}(\varepsilon\chi_k^\varepsilon)\right\}\partial_{k}u_0+a_{ij}^\varepsilon \partial_{j}\left\{\Phi_{\varepsilon,k}(x)-x_k\right\}\partial^2_{ik}u_0
\end{aligned}$$ in the second equality.
In view of $(1.8)$, then
\begin{equation}\begin{aligned}
\partial_{x_i}\left[a_{ij}^\varepsilon\partial_{x_j}\left(\varepsilon\chi_w^\varepsilon u_0\right)\right]&=\partial_{x_i}\left(a_{ij}^\varepsilon\varepsilon\chi_w^\varepsilon\partial_{x_j} u_0\right)+\partial_{x_i}\left(a_{ij}^\varepsilon\partial_{y_j}\chi_w^\varepsilon u_0\right)\\
&=\partial_{x_i}\left(a_{ij}^\varepsilon\varepsilon\chi_w^\varepsilon\partial_{x_j} u_0\right)+\frac{1}{\varepsilon}W^\varepsilon u_0+a_{ij}^\varepsilon\partial_{y_j}\chi_w^\varepsilon \partial_{x_i}u_0,
\end{aligned}\end{equation}
Thus, combining $(3.10)$ and $(3.11)$ yields that
\begin{equation}\begin{aligned}
\mathcal{L}_\varepsilon w_\varepsilon&=-\partial_{i}\left(a_{ij}^\varepsilon \partial_{j}w_\varepsilon'\right)+\partial_{i}\left[a_{ij}^\varepsilon\partial_{j}\left(\varepsilon\chi_w^\varepsilon u_0\right)\right]+\frac{1}{\varepsilon}W^\varepsilon w_\varepsilon\\
&=\varepsilon\partial_{i}\left\{F_{kij}^\varepsilon\partial^2_{kj}u_0\right\}
+a_{ij}^\varepsilon\partial_{j}\left\{\Phi_{\varepsilon,k}(x)-x_k-\varepsilon\chi_k^\varepsilon\right\}\partial^2_{ik}u_0
+\partial_{ i}\left(a_{ij}^\varepsilon\varepsilon\chi_w^\varepsilon\partial_{j} u_0\right)\\
&\quad+\partial_{i}\left\{a_{ij}^\varepsilon \left\{\Phi_{\varepsilon,k}(x)-x_k\right\}\partial^2_{jk}u_0\right\}
-\frac{1}{\varepsilon}W^\varepsilon\left\{\Phi_{\varepsilon,j}(x)-x_j-\varepsilon\chi_j^\varepsilon\right\}\partial_{j}u_0\\
&\quad+\left(a_{ij}^\varepsilon\partial_{y_j}\chi_w^\varepsilon-W^\varepsilon\chi_i^\varepsilon\right) \partial_{i}u_0
+\mathcal{M}(W\chi_w)u_0-W^\varepsilon\chi_w^\varepsilon u_0,
\end{aligned}\end{equation}
where we have used $(3.2)$ in the above equality. Consequently, in view of $(3.5)$, $(2.12)$ and $(3.6)$, there holds
\begin{equation}\begin{aligned}
\mathcal{L}_\varepsilon w_\varepsilon
&=\varepsilon\partial_{i}\left\{F_{kij}^\varepsilon\partial^2_{kj}u_0\right\}
+a_{ij}^\varepsilon\partial_{j}\left\{\Phi_{\varepsilon,k}(x)-x_k-\varepsilon\chi_k^\varepsilon\right\}\partial^2_{ik}u_0
+\partial_{i}\left(a_{ij}^\varepsilon\varepsilon\chi_w^\varepsilon\partial_{j} u_0\right)\\
&\quad+\partial_{i}\left\{a_{ij}^\varepsilon \left\{\Phi_{\varepsilon,k}(x)-x_k\right\}\partial^2_{jk}u_0\right\}
-\varepsilon\Delta_x \psi_3^\varepsilon\left\{\Phi_{\varepsilon,j}(x)-x_j-\varepsilon\chi_j^\varepsilon\right\}\partial_{j}u_0\\
&\quad+\varepsilon^2 \Delta_x \psi_{1,i}^\varepsilon \partial_{i}u_0
-\varepsilon^2 \Delta_x \psi_{2}^\varepsilon u_0,
\end{aligned}\end{equation} which is the desired equality $(3.8)$.
\end{proof}
\begin{thm}
Suppose that $A$ satisfies $(1.2)$, $(1.3)$ and $(1.4)$, and $\mathcal{M}({W\chi_w})>-\lambda_{0,1}'$. Let $\Omega$ be a convex  bounded Lipschitz domain or bounded $C^1$ domain in
$\mathbb{R}^d$. For $0\leq \varepsilon \leq \varepsilon_1$, with $\varepsilon_1$ defined in $(2.19)$, and $f\in L^2(\Omega)$, let $u_\varepsilon$ be the unique weak solution in $H^1_0(\Omega)$ to
$\mathcal{L}_\varepsilon u_\varepsilon=f$ in $\Omega$. Then
\begin{equation}
\left\|u_\varepsilon-u_0-\left\{\Phi_{\varepsilon,j}(x)-x_j\right\}\partial_{j}u_0-\varepsilon\chi_w^\varepsilon u_0\right\|_{H^1_0(\Omega)}\leq C\varepsilon ||f||_{L^2(\Omega)},
\end{equation}where $C$ depends only on $\kappa$, $\rho(t)$, $W$ and $\Omega$.
\end{thm}
\begin{proof}
Let $w_\varepsilon$ be given by $(3.7)$. It is easy to see that under the  assumptions in the Theorem 3.4, $w_\varepsilon\in H^1_0(\Omega)$.
It follows from integration by parts that
\begin{equation}\begin{aligned}
&\left|\int_\Omega\varepsilon\Delta_x \psi_3^\varepsilon\left\{\Phi_{\varepsilon,j}(x)-x_j-\varepsilon\chi_j^\varepsilon\right\}\partial_{j}u_0 w_\varepsilon dx\right|\\
\leq& C\int_\Omega\left\{\left|\nabla\left\{\Phi_{\varepsilon,j}(x)-x_j-\varepsilon\chi_j^\varepsilon\right\}\right||\nabla u_0||w_\varepsilon|+\left|\left\{\Phi_{\varepsilon,j}(x)-x_j-\varepsilon\chi_j^\varepsilon\right\}\right||\nabla^2 u_0||w_\varepsilon|\right\}dx\\
&+C\int_\Omega \left|\left\{\Phi_{\varepsilon,j}(x)-x_j-\varepsilon\chi_j^\varepsilon\right\}\right| |\nabla u_0||\nabla w_\varepsilon|dx\\
\leq &C\int_\Omega\left|\nabla\left\{\Phi_{\varepsilon,j}(x)-x_j-\varepsilon\chi_j^\varepsilon\right\}\right||\nabla u_0||w_\varepsilon|dx+C\varepsilon||u_0||_{H^2(\Omega)}||\nabla w_\varepsilon||_{L^2(\Omega)},
\end{aligned}\end{equation} where we have used $||\nabla_y \psi_{3}^\varepsilon||_{\infty}\leq C$ and $||\Phi_{\varepsilon,j}(x)-x_j-\varepsilon\chi_j^\varepsilon||_{L^\infty(\Omega)}\leq C \varepsilon$ in the above inequality.
Similarly, there holds
\begin{equation}
\left|\int_\Omega\left(\varepsilon^2 \Delta_x \psi_{1,i}^\varepsilon \partial_{x_i}u_0
-\varepsilon^2 \Delta_x \psi_{2}^\varepsilon u_0\right)w_\varepsilon dx\right|\leq C\varepsilon||u_0||_{H^2(\Omega)}||\nabla w_\varepsilon||_{L^2(\Omega)}.
\end{equation}
Consequently, it follows from $(2.21)$ and $(3.13)$ that
\begin{equation}\begin{aligned}
\int_\Omega|\nabla w_\varepsilon|^2dx\leq& C\int_\Omega\left|\nabla\left\{\Phi_{\varepsilon,j}(x)-x_j-\varepsilon\chi_j^\varepsilon\right\}\right|(|\nabla u_0|+|\nabla^2 u_0|)|w_\varepsilon|dx\\
&+ C\varepsilon ||u_0||_{H^2(\Omega)}||\nabla w_\varepsilon||_{L^2(\Omega)},
\end{aligned}\end{equation}
where we have used the estimates $||F_{kij}||_{\infty}\leq C$ of Lemma 3.1 and $||\Phi_{\varepsilon,j}(x)-x_j||_{\infty}\leq C$. We claim that
\begin{equation}
\int_\Omega\left|\nabla\left\{\Phi_{\varepsilon,j}(x)-x_j-\varepsilon\chi_j^\varepsilon\right\}\right|^2|w_\varepsilon|^2dx\leq C_0 \varepsilon^2 \int_\Omega|\nabla w_\varepsilon|^2 dx.
\end{equation}
Therefore, it follows form $(2.23)$ ,$(2.24)$, $(3.17)$ and $(3.18)$ that
\begin{equation}
||w_\varepsilon||_{H^1_0(\Omega)}\leq C||\nabla w_\varepsilon||_{L^2(\Omega)}\leq C \varepsilon ||u_0||_{H^2(\Omega)}\leq C\varepsilon||f||_{L^2(\Omega)}.
\end{equation}
To see $(3.18)$, we fix $1\leq j_0\leq d$ and let $$h_\varepsilon(x)=\Phi_{\varepsilon,j_0}(x)-x_{j_0}-\varepsilon\chi_{j_0}^\varepsilon \text{\quad in } \Omega.$$
Note that if $\Omega$ is a bounded $C^1$ domain, then $h_\varepsilon\in W^{1,p}(\Omega)\cap L^\infty(\Omega)$, with any $1<p<\infty$, due to $A\in \text{VMO}(\mathbb{R}^d)$ \cite{2008Zhongwei} and
$-\operatorname{div}(A^\varepsilon\nabla h_\varepsilon)=0 \text{ in } \Omega$, as well as $w_\varepsilon\in L^{\frac{2d}{d-2}}$ if $d\geq 2$,  and
$w_\varepsilon\in L^q(\Omega)$ for any $1<q<\infty$ if $d=2$; or if $\Omega$ is a bounded convex domain, then $w_\varepsilon\in L^\infty(\Omega)$. It follows that
\begin{equation}\begin{aligned}
\kappa \int_{\Omega }|\nabla h_\varepsilon|^2|w_\varepsilon|^2dx
&\leq \int_\Omega a_{ij}^\varepsilon \partial_{x_i}h_\varepsilon\partial_{x_j}h_\varepsilon|w_\varepsilon|^2dx\\
& =-2\int_\Omega h_\varepsilon w_\varepsilon a_{ij}^\varepsilon \partial_{x_i}h_\varepsilon\partial_{x_j}w_\varepsilon dx.
\end{aligned}\end{equation}
Hence,
\begin{equation}
\int_{\Omega }|\nabla h_\varepsilon|^2|w_\varepsilon|^2dx\leq C\int_{\Omega} |h_\varepsilon||\nabla h_\varepsilon||\nabla w_\varepsilon||w_\varepsilon|dx,
\end{equation}where $C$ depends only on $d$ and $\kappa$. Then the estimate $(3.18)$ now follows from $(3.21)$ by the Cauchy inequality and $||h_\varepsilon||_{\infty}\leq C \varepsilon$.
\end{proof}

At this position, we are ready to prove Theorem 1.1. Note that in Theorem 3.4, we don't assume that $\mathcal{L}_0$ is the effective operator of
$\mathcal{L}_\varepsilon$, although, to some extend, $\mathcal{L}_0$ is actually related to $\mathcal{L}_\varepsilon$ since we have obtained the
$H^1_0(\Omega)$ convergence rates.\\

\textbf{Proof of the Theorem 1.1:} It follows from $(2.21)$ that,
\begin{equation}
\left\langle\mathcal{L}_\varepsilon v, v\right\rangle\geq c||v||_{H^1_0(\Omega)}^2, \ c>0, \ \forall v\in H^1_0(\Omega),
\end{equation}
for $0<\varepsilon\leq \varepsilon_1$. This proves that equation admits a unique solution for $0<\varepsilon\leq \varepsilon_1$ and that
\begin{equation}
||u_\varepsilon||_{H^1_0(\Omega)}\leq C.
\end{equation}

Then there exists $u_0\in H^1_0(\Omega)$ such that $u_\varepsilon\rightharpoonup u_0$ weakly in $H^1_0(\Omega)$. In this paper, we provide a proof with an
interesting observation, and note that we don't use the so-called div-curl lemma. The natural ideal is that let $\varepsilon\rightarrow 0$ in the equation $(1.1)$, and then obtain the limiting equation.

First, we multiply the equation $(1.1)$ by $\phi\theta_\varepsilon$ and integrate the resulting equation over $\Omega$, with $\phi\in C^\infty_0(\Omega)$ and $\theta_\varepsilon=(1+\varepsilon\chi_w^\varepsilon)$, then we have
\begin{equation}
\int_\Omega a_{ij}^\varepsilon \partial_{x_j}u_\varepsilon \partial_{x_i}\phi \theta_\varepsilon+\int_\Omega a_{ij}^\varepsilon \partial_{x_j}u_\varepsilon \partial_{x_i} \theta_\varepsilon\phi+\frac{1}{\varepsilon}\int_\Omega W^\varepsilon u_\varepsilon\phi\theta_\varepsilon =\int_\Omega f\phi\theta_\varepsilon.
\end{equation}
In view of $(1.8)$ and $A^*=A$, then integration by parts yields
\begin{equation}\begin{aligned}
\int_\Omega a_{ij}^\varepsilon \partial_{x_j}u_\varepsilon \partial_{x_i} \theta_\varepsilon\phi&=\int_\Omega a_{ij}^\varepsilon \partial_{y_j} \chi_w^\varepsilon \partial_{x_i}u_\varepsilon \phi\\
&=-\int_\Omega \partial_{x_i}\left(a_{ij}^\varepsilon \partial_{y_j} \chi_w^\varepsilon \right)u_\varepsilon \phi-\int_\Omega a_{ij}^\varepsilon \partial_{y_j} \chi_w^\varepsilon \partial_{x_i}\phi u_\varepsilon \\
&=-\int_\Omega \frac{1}{\varepsilon}W^\varepsilon u_\varepsilon \phi-\int_\Omega a_{ij}^\varepsilon \partial_{y_j} \chi_w^\varepsilon \partial_{x_i}\phi u_\varepsilon.
\end{aligned}\end{equation}
Consequently, combining $(3.24)$ and $(3.25)$ gives
\begin{equation}
\int_\Omega a_{ij}^\varepsilon \partial_{x_j}u_\varepsilon \partial_{x_i}\phi \theta_\varepsilon-\int_\Omega a_{ij}^\varepsilon \partial_{y_j} \chi_w^\varepsilon \partial_{x_i}\phi u_\varepsilon+\int_\Omega W^\varepsilon\chi_w^\varepsilon u_\varepsilon\phi =\int_\Omega f\phi\theta_\varepsilon.
\end{equation}
In view of $(3.14)$, if we additional assume that $\Omega$ is a bounded $C^1$ domain or bounded convex Lipschitz domain, then there hold
\begin{equation}
u_\varepsilon=u_0+O(\varepsilon||u_0||_{H^2(\Omega)}) \text{ in }L^2(\Omega)
\end{equation}
and
\begin{equation}\begin{aligned}
&\partial_k u_\varepsilon\\
=&\partial_k u_0+\partial_k\left\{\Phi_{\varepsilon,j}(x)-x_j\right\}\partial_{j}u_0+\left\{\Phi_{\varepsilon,j}(x)-x_j\right\}\partial^2_{jk}u_0+\partial_{y_k}\chi_w^\varepsilon u_0+\varepsilon\chi_w^\varepsilon \partial_k u_0+O(\varepsilon)\\
=&\partial_k u_0+\partial_{y_k}\chi_j^\varepsilon\partial_j u_0
+\partial_k\left\{\Phi_{\varepsilon,j}(x)-x_j-\varepsilon\chi_j^\varepsilon\right\}\partial_{j}u_0+\left\{\Phi_{\varepsilon,j}(x)
-x_j\right\}\partial^2_{jk}u_0\\
&\quad+\partial_{y_k}\chi_w^\varepsilon u_0+\varepsilon\chi_w^\varepsilon \partial_k u_0+O(\varepsilon||u_0||_{H^2(\Omega)}) \text{ in }L^2(\Omega), \text{ for }k=1,\cdots,d.
\end{aligned}\end{equation}
Due to $(3.28)$, there holds
\begin{equation}\begin{aligned}
\int_\Omega a_{ij}^\varepsilon \partial_{x_j}u_\varepsilon \partial_{x_i}\phi \theta_\varepsilon=&\int_\Omega a_{ij}^\varepsilon \partial_{y_j}u_0 \partial_{x_i}\phi \theta_\varepsilon+\int_\Omega a_{ij}^\varepsilon \partial_{y_j}\chi_k^\varepsilon\partial_{x_k}u_0 \partial_{x_i}\phi\theta_\varepsilon\\
&+I_1+I_2+\int_\Omega a_{ij}^\varepsilon\partial_{y_j}\chi_w^\varepsilon u_0\partial_{x_i}\phi \theta_\varepsilon+I_3+O(\varepsilon ||\nabla \phi||_{L^2(\Omega)}||u_0||_{H^2(\Omega)}),
\end{aligned}\end{equation}
where
\begin{equation}\begin{aligned}
|I_1|^2&=\left|\int_\Omega a_{ij}^\varepsilon\partial_j\left\{\Phi_{\varepsilon,k}(x)-x_k-\varepsilon\chi_k^\varepsilon\right\}\partial_{j}u_0 \partial_{x_i}\phi\theta_\varepsilon\right|^2\\
&\leq C||\nabla u_0||^2_{L^2(\Omega)}\int_\Omega\left|\nabla\left\{\Phi_{\varepsilon,j}(x)-x_j-\varepsilon\chi_j^\varepsilon\right\}\right|^2|\nabla \phi|^2dx\\
&\leq C\varepsilon^2||\nabla u_0||^2_{L^2(\Omega)}C||\nabla^2 \phi||^2_{L^2(\Omega)},
\end{aligned}\end{equation}where we have used $(3.18)$ with $w_\varepsilon$ replaced by $\nabla\phi$ in the above inequality. Similarly, we have
\begin{equation}\begin{aligned}
|I_2|&=\left|\int_\Omega a_{ij}^\varepsilon\left\{\Phi_{\varepsilon,k}(x)-x_k\right\}\partial^2_{jk}u_0\partial_{x_i}\phi\theta_\varepsilon\right|\\
\
&\leq C\varepsilon ||\nabla^2 u_0||_{L^2(\Omega)}||\nabla \phi||_{L^2(\Omega)},
\end{aligned}\end{equation}

and
\begin{equation}
|I_3|=\varepsilon\left|\int_\Omega a_{ij}^\varepsilon\chi_w^\varepsilon \partial _{x_j}u_0\partial_{x_i}\phi \theta_\varepsilon\right|\leq C\varepsilon ||\nabla^2 u_0||_{L^2(\Omega)}||\nabla \phi||_{L^2(\Omega)}.
\end{equation}
As $\varepsilon\rightarrow 0$, combining $(3.29)$-$(3.32)$, there holds
\begin{equation}\begin{aligned}
\int_\Omega a_{ij}^\varepsilon \partial_{x_j}u_\varepsilon \partial_{x_i}\phi \theta_\varepsilon&\rightarrow \int_\Omega \mathcal{M}(a_{ij}+a_{ik} \partial_{y_k}\chi_j) \partial_{y_j}u_0 \partial_{x_i}\phi+\int_\Omega \mathcal{M}(a_{ij}\partial_{y_j}\chi_w) u_0\partial_{x_i}\phi \\
&=\int_\Omega \widehat{a}_{ij} \partial_{y_j}u_0 \partial_{x_i}\phi+\int_\Omega \mathcal{M}(a_{ij}\partial_{y_j}\chi_w) u_0\partial_{x_i}\phi.
\end{aligned}\end{equation}

Similarly, as $\varepsilon\rightarrow 0$,
\begin{equation}
-\int_\Omega a_{ij}^\varepsilon \partial_{y_j} \chi_w^\varepsilon \partial_{x_i}\phi u_\varepsilon+\int_\Omega W^\varepsilon\chi_w^\varepsilon u_\varepsilon\phi\rightarrow -\int_\Omega \mathcal{M}(a_{ij}\partial_{y_j}\chi_w) u_0\partial_{x_i}\phi+\int_\Omega \mathcal{M}(W\chi_w) u_0\phi,
\end{equation}

Therefore, combining $(3.26)$, $(3.33)$ and $(3.34)$ and letting $\varepsilon\rightarrow0$ yields
\begin{equation}
\int_\Omega \widehat{a}_{ij} \partial_{y_j}u_0 \partial_{x_i}\phi+\int_\Omega \mathcal{M}(W\chi_w) u_0\phi=\int_\Omega f\phi,
\end{equation}which is the effective equation $(1.5)$ that we want to seek.

We have completed the proof under the assumption that $\Omega$ is a bounded $C^1$ domain or bounded convex Lipschitz domain. If $\Omega$ is only a bounded Lipschitz domain, then we may need another $H^1$ convergence rate estimates with the help of $\varepsilon$-smoothing method and the first order correctors. We first give some notations. Fix a nonnegative function $\vartheta\in C_0^\infty(B(0,1/2))$ such that $\int_{\mathbb{R}^n}\vartheta dx=1$. For $\varepsilon>0,$
define \begin{equation*}
S_\varepsilon(f)(x)=\vartheta_\varepsilon\ast f(x)=\int_{\mathbb{R}^n}f(x-y)\vartheta_\varepsilon(y)dy,\end{equation*}
where $\vartheta_\varepsilon(y)=\varepsilon^{-n}\vartheta(y/\varepsilon)$. And fix a cut-off function $\eta_\varepsilon\in C^\infty_0(\Omega)$ such that
\begin{equation*}\left\{\begin{aligned}
0 \leq \eta_{\varepsilon} & \leq 1, & &\left|\nabla \eta_{\varepsilon}\right| \leq C / \varepsilon, \\
\eta_{\varepsilon}(x) &=1 & &\text { if } \operatorname{dist}(x, \partial \Omega) \geq 4 \varepsilon, \\
\eta_{\varepsilon}(x) &=0 & & \text { if } \operatorname{dist}(x, \partial \Omega) \leq 3 \varepsilon.
\end{aligned}\right.\end{equation*}
Let $S_\varepsilon^2=S_\varepsilon\circ S_\varepsilon$ and let $w_\varepsilon=u_\varepsilon-u_0-\varepsilon\chi_j^\varepsilon\eta_\varepsilon S_\varepsilon^2(\partial_j u_0)-\varepsilon\chi_w u_0$, we could obtain that
\begin{equation}||w_\varepsilon||_{H^1_0(\Omega)}\leq C\sqrt{\varepsilon}||u_0||_{H^2(\Omega)},\end{equation}
with $\Omega$ being a bounded Lipschitz domain. For the details proof of $(3.36)$, which we omit it in this paper, we may refer readers to \cite[Chapter 3.2]{shen2018periodic} for the case of $W=0$, and the proof of Lemma 3.3 to handle the nonzero potential $\frac{1}{\varepsilon}W(x/\varepsilon)$. Therefore, we may replace $(3.27)$ by the following equality
\begin{equation}
u_\varepsilon=u_0+O(\sqrt{\varepsilon}||u_0||_{H^2(\Omega)}) \text{ in }L^2(\Omega),
\end{equation} Actually, by a duality argument, we can show that $u_\varepsilon=u_0+O({\varepsilon}||u_0||_{H^2(\Omega)}) \text{ in }L^2(\Omega)$. And we may replace $(3.28)$ by the following equality,
\begin{equation}\begin{aligned}
\partial_k u_\varepsilon=\partial_k u_0+\partial_{y_k}\chi_j^\varepsilon \partial_j u_0+
\partial_{y_k}\chi_w^\varepsilon u_0+\partial_{y_k}\chi_j^\varepsilon(\eta_\varepsilon S_\varepsilon^2(\partial_ju_0)-\partial_ju_0)
+\varepsilon \chi_j^\varepsilon \partial_j \eta_\varepsilon S_\varepsilon^2(\partial_ju_0)\\
+\varepsilon \chi_j^\varepsilon  \eta_\varepsilon S_\varepsilon^2(\partial^2_{jk}u_0)
+\varepsilon\chi_w^\varepsilon \partial_k u_0+O(\sqrt{\varepsilon}||u_0||_{H^2(\Omega)}) \text{ in }L^2(\Omega), \text{ for }k=1,\cdots,d.
\end{aligned}\end{equation}
Consequently, following the proofs of $(3.29)$-$(3.34)$, we may come to the equation $(3.35)$.
\qed

\begin{rmk}To some extend, the proof above is the inverse procedure of the two-scale asymptotic expansions.
In order to make the ideal more claer, we conclude it as the following three steps.

Step 1: Due to the two-scale asymptotic expansions, we may formally obtain that $u_\varepsilon=u_0+\varepsilon\chi_j^\varepsilon\partial_{x_j}u_0+\varepsilon\chi_w^\varepsilon u_0+O(\varepsilon)$, where $\chi_j$ and $\chi_w$ are the correctors defined in $(1.7)$ and $(1.8)$, respectively, and $u_0$ satisfies the equation $(1.5)$.

Step 2: For any $f\in L^2(\Omega)$, and $u_\varepsilon\in H^1_0(\Omega)$ satisfies the equation $\mathcal{L}_\varepsilon u_\varepsilon=f$. Then there exists  $\tilde{u}\in H^2(\Omega)\cap H^1_0(\Omega)$ solving $\mathcal{L}_0 \tilde{u}=f$ (actually, under the condition $\mathcal{M}(W\chi_w)>-\lambda_1$, the solution is unique, and $\tilde{u}=u_0$). On account of the boundary condition, we need to obtain the $H^1_0$ as well as the $L^2$ convergence rates, namely, the estimates $(3.14)$.

Step 3: After obtaining the convergence rates for the $L^2$-norm and $H^1_0$-norm, we can test the equation with suitable function, and use the asymptotic expansions for $u_\varepsilon$ and $\nabla u_\varepsilon$ due to the convergence rates. Consequently, let $\varepsilon\rightarrow 0$ and then the limiting equation is exactly that we want to seek.
\end{rmk}

\section{Convergence rates for eigenvalues}
The goal of this section is to prove Theorem 1.4. For $\varepsilon_1\geq \varepsilon\geq 0$ and $f\in L^2(\Omega)$, under the conditions in Theorem 1.4,
the elliptic equation $\mathcal{L}_\varepsilon=f$ in $\Omega$ has a unique weak solution in $H^1_0(\Omega)$. Define $T_\varepsilon(f)=u_\varepsilon$.
According to $(2.21)$ and $(2.22)$, we have $||u_\varepsilon||_{H^1_0(\Omega)}\leq C ||f||_{L^2(\Omega)}$, where $C$ depends only on $A$, $W$, $d$ and $\Omega$, and then the linear operator
$T_\varepsilon$ is bounded, positive and compact on $L^2(\Omega)$. In view of $a_{ij}=a_{ji}$, the operator $T_\varepsilon$ is also self-adjoint. Let
\begin{equation}\mu_{\varepsilon, 1} \geq \mu_{\varepsilon, 2} \geq \cdots \geq \cdots>0\end{equation}
be the sequence of eigenvalues of $T_\varepsilon$ in decreasing order. By the mini-max principle,
\begin{equation}
\mu_{\varepsilon, k}=\min _{\small{f_{1}, \ldots, f_{k-1}}\atop\small{ \in L^2(\Omega)}} \max _{\small{\|f\|_{L^{2}(\Omega)}=1}, \small{f\perp f_i}\atop \small{i=1,\cdots,k-1}}\langle T_\varepsilon(f),f\rangle,
\end{equation}
where $\langle,\rangle$ denotes the inner product in $L^2(\Omega)$. Note that
\begin{equation}\left\langle T_{\varepsilon}(f), f\right\rangle=\left\langle u_{\varepsilon}, f\right\rangle=\int_{\Omega} \left(a_{i j}^\varepsilon \frac{\partial u_\varepsilon}{\partial x_{i}}\frac{\partial u_\varepsilon}{\partial x_{j}}+\frac{1}{\varepsilon}W^\varepsilon u_\varepsilon^2 \right)d x\end{equation}

and
\begin{equation}\left\langle T_{0}(f), f\right\rangle=\left\langle u_{0}, f\right\rangle=\int_{\Omega} \left(\widehat{a}_{i j} \frac{\partial u_0}{\partial x_{i}}\frac{\partial u_0}{\partial x_{j}}+\mathcal{M}({W\chi_w})u_0^2 \right)d x.\end{equation}

For $0\leq \varepsilon \leq\varepsilon_1$, let $\{\varphi_{\varepsilon,k}\}$ ba an orthonormal basis of $L^2(\Omega)$, where $\{\varphi_{\varepsilon,k}\}$ is an eigenfunctions associated with $\mu_{\varepsilon,k}$. Let $V_{\varepsilon,0}=\{0\}$ and $V_{\varepsilon,k}$ be the subspace of $L^2(\Omega)$ spanned by $\{\varphi_{\varepsilon,1},\cdots,\varphi_{\varepsilon,k}\}$ for $k\geq 1$. Then

\begin{equation}
\mu_{\varepsilon, k}=\max _{\small{f \perp V_{\varepsilon, k-1}}\atop \small{\|f\|_{L^{2}(\Omega)}=1}}\left\langle T_{\varepsilon}(f), f\right\rangle.
\end{equation}
Denote $\lambda_{\varepsilon,k}=\mu_{\varepsilon, k}^{-1}$. Then $\{\lambda_{\varepsilon,k}\}$ is the sequence of Dirichlet eigenvalues of $\mathcal{L}_\varepsilon$ in $\Omega$ in increasing order.

\begin{lemma}
Suppose that $A$ satisfies the conditions $(1.2)$-$(1.4)$ and the the symmetry condition $A^*=A$, and assume $\mathcal{M}(W\chi_w)>-\lambda_{0,1}'$. Then
\begin{equation}\left|\mu_{\varepsilon, k}-\mu_{0, k}\right| \leq \max \left\{\max _{f \perp V_{0, k-1}\atop \|f\|_{L^{2}(\Omega)}=1}\left|\left\langle\left(T_{\varepsilon}-T_{0}\right) f, f\right\rangle\right|, \max _{f \perp V_{\varepsilon, k-1}\atop \|f\|_{L^{2}(\Omega)}=1}\left|\left\langle\left(T_{\varepsilon}-T_{0}\right) f, f\right\rangle\right|\right\}\end{equation}
\end{lemma} for any $\varepsilon_1\geq\varepsilon>0$.

\begin{proof} It follows from $(4.2)$ that
\begin{equation}\begin{aligned}
\mu_{\varepsilon, k} \leq& \max _{\small{f \perp V_{0, k-1}}\atop \small{\|f\|_{L^{2}(\Omega)}=1}}\left\langle T_{\varepsilon}(f), f\right\rangle \leq \max _{\small{f \perp V_{0, k-1}}\atop \small{\|f\|_{L^{2}(\Omega)}=1}}\left\langle\left(T_{\varepsilon}-T_{0}\right)(f), f\right\rangle+\max _{\small{f \perp V_{0, k-1}}\atop \small{\|f\|_{L^{2}(\Omega)}=1}}\left\langle T_{0}(f), f\right\rangle \\
=& \max _{\small{f \perp V_{0, k-1}}\atop \small{\|f\|_{L^{2}(\Omega)}=1}}\left\langle\left(T_{\varepsilon}-T_{0}\right)(f), f\right\rangle+\mu_{0, k} ,
\end{aligned}\end{equation}
where we have used $(4.5)$ for $\varepsilon=0$. Therefore
\begin{equation}\mu_{\varepsilon, k}-\mu_{0, k} \leq \max _{\small{f \perp V_{0, k-1}} \atop \small{\|f\|_{L^{2}{(\Omega)}=1}}}\left\langle\left(T_{\varepsilon}-T_{0}\right)(f), f\right\rangle.\end{equation}

Similarly,
\begin{equation}\mu_{0, k}-\mu_{\varepsilon, k} \leq \max _{\small{f \perp V_{\varepsilon, k-1} }\atop \small{\|f\|_{L^{2}{(\Omega)}=1}}}\left\langle\left(T_{0}-T_{\varepsilon}\right)(f), f\right\rangle.\end{equation}
Consequently, the desired estimate $(4.6)$ follows readily from $(4.8)$ and $(4.9)$.
\end{proof}
It follows from $(4.6)$ that
\begin{equation}\left|\mu_{\varepsilon, k}-\mu_{0, k}\right| \leq\left\|T_{\varepsilon}-T_{0}\right\|_{L^{2} \rightarrow L^{2}}.\end{equation}
Under the assumptions in Theorem 1.4, we have $||u_\varepsilon-u_0||_{L^2(\Omega)}\leq C \varepsilon \|f\|_{L^2(\Omega)}$, where $C$ depends on $A$, $W$, $d$ and $\Omega$. Therefore, $||T_{\varepsilon}-T_{0}||_{L^{2} \rightarrow L^{2}}\leq C \varepsilon $, which implies that $|\mu_{\varepsilon,k}-\mu_{0,k}|\leq C \varepsilon$. Then
\begin{equation}\left|\lambda_{\varepsilon, k}-\lambda_{0, k}\right| \leq C \varepsilon \lambda_{0, k} \lambda_{\varepsilon, k}\leq C \varepsilon \lambda_{0, k} ^2,\end{equation}
where $C$ is independent of $\varepsilon$ and $k$. Note that the proof of $(4.11)$ relies on the convergence rates in $L^2$. The convergence estimates of $H^1_0$ in Theorem 1.3 allows us to improve the estimates by a factor of $\lambda_{0,k}^{1/2}$.\\

\textbf{Proof of Theorem 1.4}  We will use Lemma 4.1 and Theorem 1.3 to show that
\begin{equation}\left|\mu_{\varepsilon, k}-\mu_{0, k}\right| \leq C \varepsilon\left(\mu_{0, k}\right)^{1 / 2},\end{equation}
where $C$ is independent of $\varepsilon$ and $k$. Since $\lambda_{\varepsilon, k}=(\mu_{\varepsilon, k})^{-1}$ for $ \varepsilon_1\geq\varepsilon\geq 0$ and $\lambda_{\varepsilon, k}\approx \lambda_{0, k}$, this gives the desired estimate.

Let $u_\varepsilon=T_\varepsilon(f)$ and $u_0=T_0(f)$, where $||f||_{L^2(\Omega)}=1$ and $f\perp V_{0,k-1}$. In view of $(4.4)$ and $(4.5)$, we have $\langle u_0,f\rangle\leq\mu_{0,k}$. Hence, it follows from $(2.22)$ that
\begin{equation}c\left\|\nabla u_{0}\right\|_{L^{2}(\Omega)}^{2} \leq\left\langle u_{0}, f\right\rangle \leq \mu_{0, k},\end{equation}
where $c>0$ depends only on $\mathcal{M}(W\chi)$ and $\lambda_1'$. It follows that
\begin{equation}\|f\|_{H^{-1}(\Omega)} \leq C\left\|\nabla u_{0}\right\|_{L^{2}(\Omega)} +C||u_0||_{H^{-1}(\Omega)} \leq C\left(\mu_{0, k}\right)^{1 / 2}.\end{equation}

Then, \begin{equation}\begin{aligned}
\left\langle u_{\varepsilon}-u_{0}, f\right\rangle=&\left\langle u_{\varepsilon}-u_{0}-\left\{\Phi_{\varepsilon, \ell}-x_{\ell}\right\} \frac{\partial u_{0}}{\partial x_{\ell}}-\varepsilon\chi_w^\varepsilon u_0, f\right\rangle\\
&+\left\langle\left\{\Phi_{\varepsilon, \ell}-x_{\ell}\right\} \frac{\partial u_{0}}{\partial x_{\ell}}+\varepsilon\chi_w^\varepsilon u_0, f\right\rangle.\end{aligned}\end{equation}

This implies that for any $f\perp V_{0,k-1}$ with $\|f\|_{L^2(\Omega)}=1$,
\begin{equation}\begin{aligned}
\left|\left\langle u_{\varepsilon}-u_{0}, f\right\rangle\right| \leq &\left\|u_{\varepsilon}-u_{0}-\left\{\Phi_{\varepsilon, \ell}-x_{\ell}\right\} \frac{\partial u_{0}}{\partial x_{\ell}}-\varepsilon\chi_w^\varepsilon u_0\right\|_{H_{0}^{1}(\Omega)}\|f\|_{H^{-1}(\Omega)} \\
&+\left\|\left\{\Phi_{\varepsilon, \ell}-x_{\ell}\right\} \frac{\partial u_{0}}{\partial x_{\ell}}+\varepsilon\chi_w^\varepsilon u_0\right\|_{L^{2}(\Omega)}\|f\|_{L^{2}(\Omega)} \\
\leq & C \varepsilon\|f\|_{L^{2}(\Omega)}\|f\|_{H^{-1}(\Omega)}+C \varepsilon\left\|\nabla u_{0}\right\|_{L^{2}(\Omega)}\|f\|_{L^{2}(\Omega)} \\
\leq & C \varepsilon\left\|\nabla u_{0}\right\|_{L^{2}(\Omega)} \leq C \varepsilon\left(\mu_{0, k}\right)^{1 / 2},
\end{aligned}\end{equation}
where we have used $(3.14)$ in the second inequality.

Next we consider the case $||f||_{L^2(\Omega)}=1$ and $f\perp V_{\varepsilon,k-1}$.
In view of $(4.5)$, we have $\langle u_\varepsilon,f\rangle\leq\mu_{\varepsilon,k}$. Hence, it follows from $(2.21)$ that
$c\left\|\nabla u_{\varepsilon}\right\|_{L^{2}(\Omega)}^{2} \leq\left\langle u_{\varepsilon}, f\right\rangle \leq \mu_{\varepsilon, k}$.
 It follows from $(3.6)$ that
\begin{equation}\begin{aligned}
\|f\|_{H^{-1}(\Omega)}& \leq C\left\|\nabla u_{\varepsilon}\right\|_{L^{2}(\Omega)}+C||\frac{1}{\varepsilon}W^\varepsilon u_\varepsilon||_{H^{-1}(\Omega)}\\
&\leq C\left\|\nabla u_{\varepsilon}\right\|_{L^{2}(\Omega)} +C||\nabla_x(u_\varepsilon\cdot\nabla_y\psi_3^\varepsilon)||_{H^{-1}(\Omega)}+C||\nabla_y\psi_3^\varepsilon \cdot\nabla_xu_\varepsilon||_{H^{-1}(\Omega)}\\
 &\leq C\left\|\nabla u_{\varepsilon}\right\|_{L^{2}(\Omega)}  \leq C\left(\mu_{\varepsilon, k}\right)^{1 / 2}
\end{aligned}\end{equation} and it follows from $(2.22)$ that
\begin{equation}
\left\|\nabla u_{0}\right\|_{L^{2}(\Omega)} \leq C\|f\|_{H^{-1}(\Omega)} \leq C\left(\mu_{\varepsilon, k}\right)^{1 / 2},
\end{equation}where $C$ depends only on $A$, $W$ and $d$ and we have used . As before, this implies that for any $f\perp V_{\varepsilon,k-1}$ with
$||f||_{L^2(\Omega)}=1$,
\begin{equation}\begin{aligned}
\left|\left\langle u_{\varepsilon}-u_{0}, f\right\rangle\right| \leq &\left\|u_{\varepsilon}-u_{0}-\left\{\Phi_{\varepsilon, \ell}-x_{\ell}\right\} \frac{\partial u_{0}}{\partial x_{\ell}}-\varepsilon\chi_w^\varepsilon u_0\right\|_{H_{0}^{1}(\Omega)}\|f\|_{H^{-1}(\Omega)} \\
&+\left\|\left\{\Phi_{\varepsilon, \ell}-x_{\ell}\right\} \frac{\partial u_{0}}{\partial x_{\ell}}+\varepsilon\chi_w^\varepsilon u_0\right\|_{L^{2}(\Omega)}\|f\|_{L^{2}(\Omega)} \\
\leq & C \varepsilon\|f\|_{L^{2}(\Omega)}\|f\|_{H^{-1}(\Omega)}+C \varepsilon\left\|\nabla u_{0}\right\|_{L^{2}(\Omega)}\|f\|_{L^{2}(\Omega)} \\
\leq & C \varepsilon\left(\mu_{\varepsilon, k}\right)^{1 / 2} \leq C \varepsilon\left(\mu_{0, k}\right)^{1 / 2},
\end{aligned}\end{equation}
where we have used the fact that $\mu_{\varepsilon,k}\approx \mu_{0,k}$. In view of Lemma 4.1, the estimates $(4.2)$ follows from $(4.16)$ and $(4.19)$. \qed
\section{Conormal derivatives of Dirichlet eigenfunctions}
Throughout this section we always assume that $A$ satisfies $(1.2)$ and $(1.3)$, and $A$ is Lipschitz continuous as well as $\mathcal{M}(W\chi_w)>-\lambda_{0,1}'$. Let $\lambda\geq 1$ and $S_{\varepsilon,\lambda}(f)$ be defined by $(1.16)$ with $0<\varepsilon\leq \varepsilon_1$. Note that
\begin{equation}\mathcal{L}_{\varepsilon}\left(S_{\varepsilon, \lambda}(f)\right)=\lambda S_{\varepsilon, \lambda}(f)+R_{\varepsilon, \lambda}(f),\end{equation}
where \begin{equation}R_{\varepsilon, \lambda}(f)(x)=\sum_{\sqrt{\lambda_{\varepsilon, k}} \in{[\sqrt{\lambda}, \sqrt{\lambda}+1)}}\left(\lambda_{\varepsilon, k}-\lambda\right) \varphi_{\varepsilon, k}(f).\end{equation}

It is easy to see that
\begin{equation}\begin{aligned}
\left\| S_{\varepsilon, \lambda}(f)\right\|_{L^{2}(\Omega)} & \leq  \|f\|_{L^{2}(\Omega)}, \\
\left\|\nabla S_{\varepsilon, \lambda}(f)\right\|_{L^{2}(\Omega)} & \leq C \sqrt{\lambda}\|f\|_{L^{2}(\Omega)}, \\
\left\|R_{\varepsilon, \lambda}(f)\right\|_{L^{2}(\Omega)} & \leq C \sqrt{\lambda}\|f\|_{L^{2}(\Omega)}, \\
\left\|\nabla R_{\varepsilon, \lambda}(f)\right\|_{L^{2}(\Omega)} & \leq C \lambda\|f\|_{L^{2}(\Omega)},
\end{aligned}\end{equation} where $C$ depends only on the ellipticity constant of $A$, $\mathcal{M}(W\chi_w)$ and $\lambda_{0,1}'$.

\begin{lemma}
Suppose that $A$ satisfies the conditions $(1.2)$ and $(1.3)$. Also assume that $A$ is Lipschitz continuous. Let $u_\varepsilon \in H^2(\Omega)$ be a solution of $\mathcal{L}_\varepsilon(u_\varepsilon)=f$ in $\Omega$ for some $f\in L^2(\Omega)$, where $\Omega$ is a bounded Lipschitz domain. Then
\begin{equation}\begin{aligned}
\int_{\partial \Omega} n_{k} h_{k} a_{i j}^\varepsilon \frac{\partial u_{\varepsilon}}{\partial x_{i}} \cdot \frac{\partial u_{\varepsilon}}{\partial x_{j}} d \sigma=2 & \int_{\partial \Omega} h_{k}\left\{n_{k} \frac{\partial}{\partial x_{i}}-n_{i} \frac{\partial}{\partial x_{k}}\right\} u_{\varepsilon} \cdot a_{i j}^\varepsilon \frac{\partial u_{\varepsilon}}{\partial x_{j}} d \sigma \\
&-\int_{\Omega} \operatorname{div}(h) a_{i j}^\varepsilon \frac{\partial u_{\varepsilon}}{\partial x_{i}} \cdot \frac{\partial u_{\varepsilon}}{\partial x_{j}} d x \\
&-\int_{\Omega} h_{k} \frac{\partial}{\partial x_{k}}\left\{a_{i j}^\varepsilon\right\} \frac{\partial u_{\varepsilon}}{\partial x_{i}} \cdot \frac{\partial u_{\varepsilon}}{\partial x_{j}} d x \\
&+2 \int_{\Omega} \frac{\partial h_{k}}{\partial x_{i}} \cdot a_{i j}^\varepsilon \frac{\partial u_{\varepsilon}}{\partial x_{k}} \cdot \frac{\partial u_{\varepsilon}}{\partial x_{j}} d x \\
&-2 \int_{\Omega}\left( f-\frac{1}{\varepsilon}W^\varepsilon u_\varepsilon\right) \cdot \frac{\partial u_{\varepsilon}}{\partial x_{k}} \cdot h_{k} d x,
\end{aligned}\end{equation}where $h=(h_1,\cdots,h_d)\in C_0^1(\mathbb{R}^d)$ and $n$ denotes the unit outward normal to $\partial\Omega$.
\end{lemma}
\begin{proof}
The similar proof may be found in many other works by using the divergence theorem and the assumption $A^*=A$. We refer the reader to \cite{fabes1988layer} for the case of constant coefficients.
\end{proof}

\begin{lemma}
Suppose that $A$ satisfies the conditions $(1.2)$ and $(1.3)$. Also assume that $A$ is Lipschitz continuous and $\mathcal{M}(W\chi_w)>-\lambda_{0,1}'$,
and let $\Omega$ be a bounded $C^{1,1}$ domain. Let $u_\varepsilon=S_{\varepsilon,\lambda}(f)$ be defined by $(1.16)$ with $0<\varepsilon\leq
\varepsilon_1$, where $f\in L^2(\Omega)$ and $||f||_{L^2(\Omega)}=1$. Suppose that $u_\varepsilon\in H^2(\Omega)$. Then
\begin{equation}\int_{\partial \Omega}\left|\nabla u_{\varepsilon}\right|^{2} d \sigma \leq C \lambda+\frac{C}{\varepsilon} \int_{\Omega_{\varepsilon}}\left|\nabla u_{\varepsilon}\right|^{2} d x,\end{equation}
where $\Omega_\varepsilon=\{x\in \Omega :\text{ dist}(x,\partial\Omega)<\varepsilon\}$ and $C$ depends only on $A$, $W$ and $\Omega$.
\end{lemma}

\begin{proof}
The similar proof can be found in \cite{kenig2013estimates}, in the case of $W=0$. Assume first that $\varepsilon\leq \text{diam}(\Omega)$.
In this case, we may choose the vector field $h\in C^1_0(\mathbb{R}^d)$ such that $n_kh_k\geq c>0$ on $\partial\Omega$, $|h|\leq 1$, $|\nabla h|\leq
C\varepsilon^{-1}$, and $h=0$ on $\{x\in \Omega:\text{ dist}(x,\partial\Omega)\geq c\varepsilon\}$, where $c=c(\Omega)>0$ is small. Note that
$\mathcal{L}_{\varepsilon}\left(S_{\varepsilon, \lambda}(f)\right)=\lambda S_{\varepsilon, \lambda}(f)+R_{\varepsilon, \lambda}(f)$ in $\Omega$. Since
$u_\varepsilon=0$ on $\partial\Omega$, it follows from $(5.4)$ that
\begin{equation}\begin{aligned}
c \int_{\partial \Omega}\left|\nabla u_{\varepsilon}\right|^{2} d \sigma \leq & \frac{C}{\varepsilon} \int_{\Omega_{\varepsilon}}\left|\nabla u_{\varepsilon}\right|^{2} d x-2 \lambda \int_{\Omega} u_{\varepsilon} \frac{\partial u_{\varepsilon} }{\partial x_{k}}h_{k} d x \\
&-2 \int_{\Omega}\left(R_{\varepsilon, \lambda}(f)\right)   \frac{\partial u_{\varepsilon} }{\partial x_{k}} h_{k} d x+2\int_\Omega \frac{1}{\varepsilon}
W^\varepsilon u_\varepsilon \frac{\partial u_{\varepsilon} }{\partial x_{k}} h_{k} d x.
\end{aligned}\end{equation}

Using the Cauchy inequality we may bound the third integral on the RHS of $(5.6)$ by $C||R_{\varepsilon,\lambda}(f)||_{L^2(\Omega)}||\nabla u_\varepsilon||_{L^2(\Omega)}$, which, in view of $(5.3)$, is dominated by $C\lambda$.
To handle the second term on the RHS of $(5.6)$, integration by parts yields
\begin{equation}\left|2 \lambda \int_{\Omega} u_{\varepsilon} \cdot \frac{\partial u_{\varepsilon}}{\partial x_{k}} \cdot h_{k} d x\right|=\left|\lambda \int_{\Omega}| u_{\varepsilon}|^{2} \operatorname{div}(h) d x\right|\leq \frac{C \lambda}{\varepsilon} \int_{\Omega_{c \varepsilon}}| u_{\varepsilon}|^{2} d x.\end{equation}
Since
\begin{equation}\begin{aligned}
\lambda\left|u_{\varepsilon}\right|^{2}-a_{i j}^\varepsilon \frac{\partial u_{\varepsilon}}{\partial x_{i}} \frac{\partial u_{\varepsilon}}{\partial x_{j}}
&=\left(\lambda u_{\varepsilon}-\mathcal{L}_{\varepsilon}\left(u_{\varepsilon}\right)\right) u_{\varepsilon}-\frac{\partial}{\partial x_{i}}\left\{u_{\varepsilon} a_{i j}^\varepsilon \frac{\partial u_{\varepsilon}}{\partial x_{j}}\right\}+\frac{1}{\varepsilon}W^\varepsilon u_\varepsilon^2\\
&=\left(\lambda u_{\varepsilon}-\mathcal{L}_{\varepsilon}\left(u_{\varepsilon}\right)\right) u_{\varepsilon}-\frac{\partial}{\partial x_{i}}\left\{u_{\varepsilon} a_{i j}^\varepsilon \frac{\partial u_{\varepsilon}}{\partial x_{j}}\right\}\\
&\quad+\nabla_x(u_\varepsilon^2\cdot\nabla_y\psi_3^\varepsilon)
-\nabla_y\psi_3^\varepsilon \cdot\nabla_xu_\varepsilon^2,
\end{aligned}\end{equation}where we have used $\Delta_y \psi_3(y)=W(y)$ in $\mathbb{R}^d$ in the above equation. Then, it follows that for any $\varphi\in C^1_0(\mathbb{R}^d)$,
\begin{equation}\begin{aligned}
\int_\Omega\left\{\lambda\left|u_{\varepsilon}\right|^{2}-a_{i j}^\varepsilon \frac{\partial u_{\varepsilon}}{\partial x_{i}} \frac{\partial u_{\varepsilon}}{\partial x_{j}}\right\}\varphi^2 dx=&\int_\Omega \left(\lambda u_{\varepsilon}-\mathcal{L}_{\varepsilon}\left(u_{\varepsilon}\right)\right) u_{\varepsilon}\varphi^2dx -2\int_\Omega u_\varepsilon^2\varphi \nabla_y\psi_3^\varepsilon\nabla \varphi dx\\
&\quad+2\int_\Omega  u_{\varepsilon}\varphi a_{i j}^\varepsilon \frac{\partial u_{\varepsilon}}{\partial x_{j}}\frac{\partial\varphi}{\partial_{x_i}}
-2\int_\Omega u_\varepsilon\varphi^2\nabla_y\psi_3^\varepsilon \cdot\nabla_xu_\varepsilon dx.
\end{aligned}\end{equation}
Choose $\varphi$ so that $0\leq \varphi\leq 1$, $\varphi(x)=1$ if $\text{dist}(x,\partial\Omega)\leq c\varepsilon$, $\varphi(x)=0$ if $\text{dist}(x,\partial\Omega)\geq 2c\varepsilon$, and $|\nabla \varphi|\leq C\varepsilon^{-1}$. In view of $(5.3)$ and $(5.9)$, then
\begin{equation}\begin{aligned}
&\lambda \int_{\Omega}\left|u_{\varepsilon}\right|^{2} \varphi^{2} d x \\
 \leq& C \int_{\Omega}\left|\nabla u_{\varepsilon}\right|^{2} \varphi^{2} d x+\int_{\Omega}\left|R_{\varepsilon,
 \lambda}(f)\right|\left|u_{\varepsilon}\right| \varphi^{2} d x+C \int_{\Omega}\left|u_{\varepsilon}\right|^{2}|\nabla \varphi|^{2} d x +C
 \int_{\Omega}\left|u_{\varepsilon}\right|^{2} \varphi^{2} d x\\
 \leq &C \int_{\Omega}\left|\nabla u_{\varepsilon}\right|^{2} \varphi^{2} d x+\left\|R_{\varepsilon,
 \lambda}(f)\right\|_{L^{2}(\Omega)}\left\|u_{\varepsilon}\right\|_{L^{2}\left(\Omega_{2 c \varepsilon}\right)}+\frac{C}{\varepsilon^{2}} \int_{\Omega_{2 c \varepsilon}}\left|u_{\varepsilon}\right|^{2} d x \\
 \leq& C \int_{\Omega_{2 c \varepsilon}}\left|\nabla u_{\varepsilon}\right|^{2} d x+C \varepsilon \lambda,
\end{aligned}\end{equation}
where we have used the Cauchy inequality, $\lambda\geq 1$, $||\nabla_y\psi_3||_{\infty}\leq C$, and the following inequality

\begin{equation}\int_{\Omega_{2 c \varepsilon}}\left|u_{\varepsilon}\right|^{2} d x \leq C \varepsilon^{2} \int_{\Omega_{2 c \varepsilon}}\left|\nabla u_{\varepsilon}\right|^{2} d x.\end{equation}
Thus, together with $(5.7)$, gives that the second term on the RHS of $(5.6)$ satisfies the desired estimate $(5.5)$. Now, we need to handle the fourth term on the RHS of $(5.6)$.
In view of $\Delta_y \psi_3(y)=W(y)$ in $\mathbb{R}^d$, then
\begin{equation}\begin{aligned}
-&\int_\Omega\frac{1}{\varepsilon}W^\varepsilon u_\varepsilon \partial_k u_\varepsilon h_kdx =-\int_\Omega\varepsilon (\Delta_x \psi_3^\varepsilon) u_\varepsilon \partial_k u_\varepsilon h_kdx\\
&=\int_\Omega\nabla_y \psi_3^\varepsilon \nabla_xu_\varepsilon \partial_k u_\varepsilon h_kdx+\int_\Omega\nabla_y \psi_3^\varepsilon u_\varepsilon \nabla_x\partial_k u_\varepsilon h_kdx+\int_\Omega\nabla_y \psi_3^\varepsilon u_\varepsilon \partial_k u_\varepsilon \nabla_xh_kdx,
\end{aligned}\end{equation} which gives

\begin{equation}\begin{aligned}
\left|\int_\Omega\frac{1}{\varepsilon}W^\varepsilon u_\varepsilon \partial_k u_\varepsilon h_kdx\right|\leq& C \int_{\Omega_{ c \varepsilon}}\left|\nabla u_{\varepsilon}\right|^{2} d x+C||u_\varepsilon||_{L^2(\Omega_{ c \varepsilon})}||\nabla^2 u_\varepsilon||_{L^2(\Omega_{ c \varepsilon})}\\
&+\frac{C}{\varepsilon}||u_\varepsilon||_{L^2(\Omega_{ c \varepsilon})}||\nabla u_\varepsilon||_{L^2(\Omega_{ c \varepsilon})}\\
\leq& C ||\nabla u_\varepsilon||_{L^2(\Omega_{ c \varepsilon})}^2+C||u_\varepsilon||_{L^2(\Omega_{ c \varepsilon})}||\nabla^2 u_\varepsilon||_{L^2(\Omega_{ c \varepsilon})},
\end{aligned}\end{equation}
where we have used $(5.11)$. In order to estimate $||\nabla^2 u_\varepsilon||_{L^2(\Omega_{ c \varepsilon})}$, we first give some notations. Let
$\tilde{\Omega}=\{x:\varepsilon x\in \Omega\}$, and $\tilde{\Omega}_c=\{x:\text{dist}(x,\partial \tilde{\Omega})<c\}$,  then it is easy to check that
$\varepsilon\tilde{\Omega}_c=\Omega_{c\varepsilon}$. If we let $v(x)=u_\varepsilon(\varepsilon x)$ and $\tilde{g}(x)=\varepsilon^2 (\lambda u_\varepsilon+R_{\varepsilon,\lambda}(f))(\varepsilon x)$, then $v$ satisfies

\begin{equation*}\left\{\begin{aligned}-\operatorname{div}(A\nabla v)+\varepsilon W v&=\tilde{g} \text{ in }\tilde{\Omega}\\
v&=0 \text{ on } \partial \tilde{\Omega}.\end{aligned}\right.
\end{equation*} Thus, according to the $W^{2,p}$ estimates with $\Omega$ being a bounded $C^{1,1}$ domain,
\begin{equation}
||\nabla^2 v||_{L^2(\tilde{\Omega}_c)}\leq C\left\{|| v||_{L^2(\tilde{\Omega}_{2c})}+||\nabla v||_{L^2(\tilde{\Omega}_{2c})}+|| \tilde{g}||_{L^2(\tilde{\Omega}_{2c})}\right\},
\end{equation} which, consequently, gives
\begin{equation}\begin{aligned}
||\nabla^2 u_\varepsilon||_{L^2({\Omega}_{c\varepsilon})}&\leq C\left\{\varepsilon^{-2}|| u_\varepsilon||_{L^2({\Omega}_{2c\varepsilon})}+\varepsilon^{-1}||\nabla u_\varepsilon||_{L^2({\Omega}_{2c\varepsilon})}+||\lambda u_\varepsilon+R_{\varepsilon,\lambda}(f) ||_{L^2(\tilde{\Omega}_{2c\varepsilon})}\right\}\\
&\leq C \lambda + C\varepsilon^{-1}||\nabla u_\varepsilon||_{L^2({\Omega}_{2c\varepsilon})}.
\end{aligned}\end{equation}
Combining $(5.13)$ and $(5.15)$ gives
\begin{equation}\left|\int_\Omega\frac{1}{\varepsilon}W^\varepsilon u_\varepsilon \partial_k u_\varepsilon h_kdx\right|\leq C \lambda + C||\nabla u_\varepsilon||^2_{L^2({\Omega}_{2c\varepsilon})}.\end{equation}
Consequently, the fourth term on the RHS of $(5.6)$ also satisfies the desired estimates $(5.5)$.

Finally, if $\varepsilon\geq \text{diam}(\Omega)$, we choose a vector field $h\in C^1_0(\mathbb{R}^d)$ such that $h_kn_k\geq c>0$ on $\partial\Omega$. The same argument as in $(5.6)$, $(5.7)$ and $(5.16)$ shows that the LHS of $(5.5)$ is bounded by $C\lambda$.

\end{proof}

\begin{thm}
Suppose that $A$ satisfies conditions $(1.2)$ and $(1.3)$. Also assume that $A$ is Lipschitz continuous and $\mathcal{M}(W\chi_w)>-\lambda_{0,1}'$. Let $\Omega$ be a bounded $C^{1,1}$ domain. Let $u_\varepsilon=S_{\varepsilon,\lambda}(f)$ be defined by $(1.16)$ with $0<\varepsilon \leq\varepsilon_1$, where $f\in L^2(\Omega)$ and $||f||_{L^2(\Omega)}=1$. Then

\begin{equation}\int_{\partial \Omega}\left|\nabla u_{\varepsilon}\right|^{2} d \sigma \leq\left\{\begin{array}{ll}
C \lambda\left(1+\varepsilon^{-1}\right) & \text {if } \varepsilon^{2} \lambda \geq 1, \\
C \lambda(1+\varepsilon \lambda) & \text { if } \varepsilon^{2} \lambda<1,
\end{array}\right.\end{equation} where $C$ depends only on $A$, $W$ and $\Omega$.
\end{thm}
\begin{proof}
We may use Lemma 5.2 and reduce the problem to the estimate of $\varepsilon^{-1}||\nabla u_\varepsilon||_{L^2(\Omega_{\varepsilon})}^2$ on the RHS of $(5.5)$.
If $\varepsilon^2\lambda\geq 1$, the desired estimate $(5.17)$ follows directly from $||\nabla u_\varepsilon||_{L^2(\Omega_{\varepsilon})}^2\leq C\lambda$.

The proof for the case $\varepsilon^2\lambda<1$ uses the $H^1$ convergence estimate in Theorem 1.3. Let $v_\varepsilon$ be the unique solution in $H^1_0(\Omega)$ to the equation
\begin{equation}\mathcal{L}_{0}\left(v_{\varepsilon}\right)=\lambda u_{\varepsilon}+R_{\varepsilon, \lambda}(f) \quad \text { in }\Omega.\end{equation}

And observe that
 \begin{equation}\left\|\lambda u_{\varepsilon}+R_{\varepsilon, \lambda}(f)\right\|_{L^{2}(\Omega)} \leq C \lambda.\end{equation}
First, using $\mathcal{L}_0v_\varepsilon=\mathcal{L}_\varepsilon u_\varepsilon$ in $\Omega$, we may deduce that

\begin{equation*}\begin{aligned}
c||\nabla v_\varepsilon||_{L^{2}(\Omega)}^2&\leq \langle\mathcal{L}_{0}\left(v_{\varepsilon}\right), v_\varepsilon\rangle=\langle\mathcal{L}_\varepsilon u_\varepsilon,v_\varepsilon\rangle\\
&\leq C||\nabla v_\varepsilon||_{L^{2}(\Omega)}||\nabla u_\varepsilon||_{L^{2}(\Omega)}+\int_\Omega \frac{1}{\varepsilon}W^\varepsilon u_\varepsilon v_\varepsilon dx\\
&\leq C||\nabla v_\varepsilon||_{L^{2}(\Omega)}||\nabla u_\varepsilon||_{L^{2}(\Omega)}+\int_\Omega \nabla_x(\nabla_y \psi_3^\varepsilon) u_\varepsilon v_\varepsilon dx\\
&\leq C||\nabla v_\varepsilon||_{L^{2}(\Omega)}||\nabla u_\varepsilon||_{L^{2}(\Omega)},
\end{aligned}\end{equation*}
where we have used $\Delta_y \psi_3(y)=W(y)$ in $\mathbb{R}^d$ and $||\nabla_y\psi_3||_{\infty}\leq C$ in the above inequality.
Thus, in view of $(5.3)$,
\begin{equation}
||\nabla v_\varepsilon||_{L^{2}(\Omega)}\leq C||\nabla u_\varepsilon||_{L^{2}(\Omega)}\leq C\sqrt{\lambda}.
\end{equation}
Then, according to the $W^{2,p}$ estimates for the constant coefficients,
\begin{equation}
||\nabla^2 v_\varepsilon||_{L^{2}(\Omega)} \leq C \left(\left\|\lambda u_{\varepsilon}+R_{\varepsilon, \lambda}(f)\right\|_{L^{2}(\Omega)} +||\mathcal{M}(W\chi_w) v_\varepsilon||_{L^{2}(\Omega)}\right)\leq C\lambda.
\end{equation}
To estimate $\varepsilon^{-1}||\nabla u_\varepsilon||_{L^2(\Omega_{\varepsilon})}^2$, we use the $H^1$ convergence estimates $(3.14)$ to obtain
\begin{equation}\begin{aligned}
\frac{1}{\varepsilon}\int_{\Omega_\varepsilon}\left|\partial_{x_i}u_\varepsilon\right|^2dx\leq& \frac{C}{\varepsilon}\int_{\Omega_\varepsilon}
\left|\partial_{x_i}\left(u_{\varepsilon}-v_\varepsilon-\left\{\Phi_{\varepsilon, \ell}-x_{\ell}\right\} \partial_{x_\ell}
v_\varepsilon-\varepsilon\chi_w^\varepsilon v_\varepsilon\right)\right|^2dx\\
&+\frac{C}{\varepsilon}\int_{\Omega_\varepsilon} |\nabla v_\varepsilon|^2dx+C\varepsilon \int_{\Omega_\varepsilon} |\nabla^2 v_\varepsilon|^2dx\\
\leq & C\varepsilon \lambda^2+\frac{C}{\varepsilon}\int_{\Omega_\varepsilon} |\nabla v_\varepsilon|^2dx,
\end{aligned}\end{equation}
 where we have used $||\nabla \Phi_\varepsilon||_\infty\leq C$ in \cite{avellaneda1987compactness}, $||\nabla_y \chi_w||_\infty+||\chi_w||_\infty\leq C$,
 due to $\chi_w\in C^{1,\alpha}$ under the assumption that $A$ is Lipschitz, $||\Phi_{\varepsilon,j}-x_j||_\infty\leq C\varepsilon$ as well as $(5.19)$ and $(3.14)$.
 Furthermore, we may use the Fundamental Theorem of Calculus to obtain
 \begin{equation}\begin{aligned}
\frac{1}{\varepsilon} \int_{\Omega_{\varepsilon}}\left|\nabla v_{\varepsilon}\right|^{2} d x & \leq C \int_{\partial \Omega}\left|\nabla v_{\varepsilon}\right|^{2} d \sigma+C \varepsilon \int_{\Omega_{\varepsilon}}\left|\nabla^{2} v_{\varepsilon}\right|^{2} d x \\
& \leq C \int_{\partial \Omega}\left|\nabla v_{\varepsilon}\right|^{2} d \sigma+C \varepsilon \lambda^{2},
\end{aligned}\end{equation}
 where we have used $(5.21)$ for the second inequality. As a result, it suffices to show that
\begin{equation}\int_{\partial \Omega}\left|\nabla v_{\varepsilon}\right|^{2} d \sigma \leq C \lambda(1+\varepsilon \lambda).\end{equation}
To this end we use a Rellich identity for $\mathcal{L}_0$, similar to $(5.4)$ for $\mathcal{L}_\varepsilon$.
\begin{equation}\begin{aligned}
\int_{\partial \Omega}\left|\nabla v_{\varepsilon}\right|^{2} d x & \leq C \int_{\Omega}\left|\nabla v_{\varepsilon}\right|^{2} d x+C\left|\int_{\Omega}\left\{\lambda u_{\varepsilon}+R_{\varepsilon, \lambda}(f)-\mathcal{M}(W\chi_w)v_\varepsilon\right\} \cdot \frac{\partial v_{\varepsilon}}{\partial x_{k}} \cdot h_{k} d x\right| \\
& \leq C \lambda+C \lambda\left|\int_{\Omega} u_{\varepsilon} \cdot \frac{\partial v_{\varepsilon}}{\partial x_{k}} \cdot h_{k} d x\right| \\
& \leq C \lambda+C \lambda\left|\int_{\Omega} v_{\varepsilon} \cdot \frac{\partial u_{\varepsilon}}{\partial x_{k}} \cdot h_{k} d x\right|+C\lambda\left|\int_{\Omega} u_{\varepsilon} \cdot v_{\varepsilon} \cdot \operatorname{div}(h) d x\right|,
\end{aligned}\end{equation}
where $h=(h_1,\cdots,h_d)\in C^\infty_0(\mathbb{R}^d)$ is a vector field such that $h_kn_k\geq c>0$ on $\partial\Omega$ and $|h|+|\nabla h|+|\nabla^2h|\leq C$.
Note that $\frac{1}{\varepsilon}W^\varepsilon u_\varepsilon=\varepsilon(\Delta_x\psi_3^\varepsilon)u_\varepsilon=\nabla_x(\nabla_y \psi_3^\varepsilon u_\varepsilon)-\nabla_y \psi_3^\varepsilon \nabla_x u_\varepsilon$, and $\mathcal{L}_{\varepsilon}\left(S_{\varepsilon, \lambda}(f)\right)=\lambda S_{\varepsilon, \lambda}(f)+R_{\varepsilon, \lambda}(f)$ in $\Omega$, then
\begin{equation*}\begin{aligned}
\left\|u_{\varepsilon}\right\|_{H^{-1}(\Omega)} &=\lambda^{-1}\left\|\mathcal{L}_{\varepsilon}\left(u_{\varepsilon}\right)-R_{\varepsilon, \lambda}(f)\right\|_{H^{-1}(\Omega)} \\
& \leq C \lambda^{-1}\left(\left\|\nabla u_{\varepsilon}\right\|_{L^{2}(\Omega)}+\left\|R_{\varepsilon, \lambda}(f)\right\|_{L^{2}(\Omega)}
+|| u_\varepsilon||_{L^{2}(\Omega)}\right)\\
& \leq C \lambda^{-1 / 2},
\end{aligned}\end{equation*}
where we have used $||\nabla_y \chi_w||_\infty\leq C$ and $(5.3)$. It follows that
\begin{equation}\lambda\left|\int_{\Omega} u_{\varepsilon} v_{\varepsilon} \operatorname{div}(h) d x\right| \leq C \lambda\left\|u_{\varepsilon}\right\|_{H^{-1}(\Omega)}\left\|v_{\varepsilon} \operatorname{div}(h)\right\|_{H_{0}^{1}(\Omega)} \leq C \lambda.\end{equation}
Integration by parts gives
\begin{equation}\begin{aligned}
\left|\int_{\Omega} v_{\varepsilon} \frac{\partial u_{\varepsilon}}{\partial x_{k}} h_{k} d x\right| \leq &\left|\int_{\Omega}\left(u_{\varepsilon}-v_{\varepsilon}\right) \frac{\partial u_{\varepsilon}}{\partial x_{k}} h_{k} d x\right|+\frac{1}{2}\left|\int_{\Omega}| u_{\varepsilon}|^{2} \operatorname{div}(h) d x\right|\\
\leq &\left|\int_{\Omega}\left\{u_{\varepsilon}-v_{\varepsilon}-\left\{\Phi_{\varepsilon, j}-x_{j} \right\} \frac{\partial v_{\varepsilon}}{\partial x_{j}}-\varepsilon\chi_w v_\varepsilon\right\} \frac{\partial u_{\varepsilon}}{\partial x_{k}} h_{k} d x\right| \\
&+\left|\int_{\Omega}\left(\left\{\Phi_{\varepsilon, j}-x_{j}\right\} \frac{\partial v_{\varepsilon}}{\partial x_{j}} +\varepsilon\chi_w v_\varepsilon\right) \frac{\partial u_{\varepsilon}}{\partial x_{k}} h_{k} d x\right|+C \\
\leq & C\left\|\left(\nabla u_{\varepsilon}\right) h\right\|_{H^{-1}(\Omega)}\left\|u_{\varepsilon}-v_{\varepsilon}-\left\{\Phi_{\varepsilon, j}-x_{j}\right\} \frac{\partial v_{\varepsilon}}{\partial x_{j}}-\varepsilon\chi_w v_\varepsilon\right\|_{H_{0}^{1}(\Omega)} \\
&+C \varepsilon\left\|\nabla u_{\varepsilon}\right\|_{L^{2}(\Omega)}\left\|\nabla v_{\varepsilon}\right\|_{L^{2}(\Omega)}+C \\
& \leq C+C \varepsilon \lambda,
\end{aligned}\end{equation} where we have used Theorem 1.3 as well as the estimates $(5.3)$, $(5.19)$ and $(5.20)$ for the last inequality. This completes the proof after combining $(5.25)$-$(5.27)$.
\end{proof}
\section{Lower bound}
In this section we give the proof of Theorem 1.6. Throughout this section we will assume that $\Omega$ is a bounded $C^2$ domain in $\mathbb{R}^d$, $d\geq 2$. We will also assume that $A$ satisfies the conditions $(1.2)$ and $(1.3)$, and $A$ is Lipschitz continuous as well as $\mathcal{M}(W\chi_w)>-\lambda_{0,1}'$.

Recall that $\Phi_\varepsilon(x)=(\Phi_{\varepsilon,i}(x))_{1\leq i\leq d}$ denotes the Dirichlet correctors for $\mathcal{L}_\varepsilon'$ in $\Omega$.
The following result states that the matrix $(\Phi_{\varepsilon,i}(x))_{1\leq i\leq d}$ is invertible near the boundary, whose proof can be founded in \cite[Lemma 5.1]{kenig2013estimates}.
\begin{lemma}
Let $J(\Phi_\varepsilon)$ denote the absolute value of the determinant of the $d\times d$ matrix $(\partial \Phi_\varepsilon/\partial{x_j})$. Then there exists constants $\varepsilon_2>0$ and $c>0$, depending only on $A$ and $\Omega$, such that for $0<\varepsilon<\varepsilon_2$,
$$J(\Phi_\varepsilon)(x)\geq c,\text{\quad if }x\in \Omega \text{ and dist}(x,\partial\Omega)\leq c\varepsilon.$$
\end{lemma}

\begin{lemma}
Let $u_\varepsilon$ be a Dirichlet eigenfunction for $\mathcal{L}_\varepsilon$ in $\Omega$ with the associated eigenvalue $\lambda$ and $||u_\varepsilon||_{L^2(\Omega)}=1$. Then, if $0<\varepsilon<\min\{\varepsilon_1,\varepsilon_2\}$,
\begin{equation}\frac{1}{\varepsilon} \int_{\Omega_{c \varepsilon}}\left|\nabla u_{\varepsilon}\right|^{2} d x \geq c \lambda-C \varepsilon \lambda^{2},
\end{equation}where $c>0$ and $C>0$ depend only on $A$, $W$, $d$ and $\Omega$.
\end{lemma}
\begin{proof}
Let $v_\varepsilon$ be the unique solution in $H^1_0(\Omega)$ to the equation $\mathcal{L}_0(v_\varepsilon)=\lambda u_\varepsilon$ in $\Omega$. As in the proof of Theorem $(5.3)$, we have $||\nabla v_\varepsilon||_{L^2(\Omega)}\leq C\sqrt{\lambda}$ and $||\nabla^2 v_\varepsilon||_{L^2(\Omega)}\leq C\lambda$.
Moreover, it follows from $(3.14)$ that
\begin{equation}
||\nabla u_\varepsilon-\left(\nabla \Phi_{\varepsilon}(x)\right)\nabla v_\varepsilon-\{\Phi_{\varepsilon}(x)-x\}\nabla^2 v_\varepsilon-\varepsilon\chi_w^ \varepsilon \nabla v_\varepsilon-\nabla_y\chi_w^ \varepsilon v_\varepsilon||_{L^2(\Omega_{c\varepsilon})}^2\leq C \varepsilon ^2\lambda^2.
\end{equation}
Hence,
\begin{equation}\begin{aligned}
\frac{1}{\varepsilon} \int_{\Omega_{c\varepsilon}}\left|\nabla u_{\varepsilon}\right|^{2} d x \geq& \frac{1}{2 \varepsilon} \int_{\Omega_{c\varepsilon}}\left|\left(\nabla \Phi_{\varepsilon}\right) \nabla v_{\varepsilon}\right|^{2} d x-C\varepsilon\int_{\Omega_{c\varepsilon}}\left(\left|\nabla^2 v_{\varepsilon}\right|^{2}+\left|\nabla v_{\varepsilon}\right|^{2}\right) d x\\
&\quad-\frac{1}{\varepsilon} \int_{\Omega_{c\varepsilon}}\left|v_{\varepsilon}\right|^{2} d x-C \varepsilon \lambda^{2}\\
 &\geq \frac{c}{\varepsilon} \int_{\Omega_{c\varepsilon}}\left|\nabla v_{\varepsilon}\right|^{2} d x-C \varepsilon \lambda^{2},
\end{aligned}\end{equation}
where we used Lemma 6.1 and the Poincar\'{e} inequality $||v_\varepsilon||_{L^2(\Omega_{c\varepsilon})}\leq C \varepsilon ||\nabla v_\varepsilon||_{L^2(\Omega_{c\varepsilon})}$.  Using
\begin{equation}\int_{\partial \Omega}\left|\nabla v_{\varepsilon}\right|^{2} d \sigma \leq \frac{C}{\varepsilon} \int_{\Omega_{c \varepsilon}}\left|\nabla v_{\varepsilon}\right|^{2} d x+C \varepsilon \int_{\Omega_{c \varepsilon}}\left|\nabla^{2} v_{\varepsilon}\right|^{2} d x\end{equation} and $||\nabla^2 v_\varepsilon||_{L^2(\Omega)}\leq C \lambda$, we further obtain
\begin{equation}\frac{1}{\varepsilon} \int_{\Omega_{c \varepsilon}}\left|\nabla u_{\varepsilon}\right|^{2} d x \geq c \int_{\partial \Omega}\left|\nabla v_{\varepsilon}\right|^{2} d \sigma-C \varepsilon \lambda^{2}.\end{equation} We claim that
\begin{equation}\lambda \leq C \int_{\partial \Omega}\left|\nabla v_{\varepsilon}\right|^{2} d \sigma+C \varepsilon \lambda^{2},\end{equation}
then combining $(6.5)$ and $(6.6)$ gives the desired estimates $(6.1)$. To see $(6.6)$, without loss of generality, assume that $0\in \Omega$. It follows by taking $h(x)=x$ in a Rellich identity for $\mathcal{L}_0$, similar to $(5.4)$ that,
\begin{equation}\begin{aligned}
&\int_{\partial \Omega}\langle x, n\rangle \hat{a}_{i j} \frac{\partial v_{\varepsilon}}{\partial x_{j}}\frac{\partial v_{\varepsilon}}{\partial x_{i}} d \sigma \\
=&(2-d) \int_{\Omega} \hat{a}_{i j} \frac{\partial v_{\varepsilon}}{\partial x_{j}}\frac{\partial v_{\varepsilon}}{\partial x_{i}} d x-2 \lambda
\int_{\Omega} u_{\varepsilon} \frac{\partial v_{\varepsilon}}{\partial x_{k}} x_{k} d x +2\mathcal{M}(W\chi_w)\int_{\Omega} v_{\varepsilon} \frac{\partial v_{\varepsilon}}{\partial x_{k}} x_{k} d x\\
=&(2-d) \lambda \int_{\Omega} u_{\varepsilon} v_{\varepsilon} d x-2 \lambda \int_{\Omega} u_{\varepsilon} \frac{\partial v_{\varepsilon}}{\partial x_{k}}
x_{k} d x-2\mathcal{M}(W\chi_w)\int_{\Omega} v_{\varepsilon} ^2d x,
\end{aligned}\end{equation}
where we have used the following equality
$$\begin{aligned}
(2-d) \int_{\Omega} \hat{a}_{i j} \frac{\partial v_{\varepsilon}}{\partial x_{j}}\frac{\partial v_{\varepsilon}}{\partial x_{i}} d x=(2-d)\lambda \int_{\Omega}u_\varepsilon v_\varepsilon dx- (2-d)\mathcal{M}(W\chi_w)\int_{\Omega} v_\varepsilon^2 dx,
\end{aligned}$$
obtained by multiplying the equation $\mathcal{L}_0(v_\varepsilon)=\lambda u_\varepsilon$ by $v_\varepsilon$ and integrating the resulting equation over $\Omega$.
It is easy to see that
\begin{equation*}\begin{aligned}
2 \int_{\Omega} u_{\varepsilon} \frac{\partial v_{\varepsilon}}{\partial x_{k}} x_{k} d x &=-2 \int_{\Omega} \frac{\partial u_{\varepsilon}}{\partial x_{k}} v_{\varepsilon} x_{k} d x-2 d \int_{\Omega} u_{\varepsilon} v_{\varepsilon} d x \\
&=d-2 \int_{\Omega} \frac{\partial u_{\varepsilon}}{\partial x_{k}}\left(v_{\varepsilon}-u_{\varepsilon}\right) x_{k} d x-2 d \int_{\Omega} u_{\varepsilon} v_{\varepsilon} d x.
\end{aligned}\end{equation*}
Then, \begin{equation}\begin{aligned}
&\int_{\partial \Omega}\langle x, n\rangle \hat{a}_{i j} \frac{\partial v_{\varepsilon}}{\partial x_{j}} \cdot \frac{\partial v_{\varepsilon}}{\partial x_{i}} d \sigma \\
&=2 \lambda+(d+2) \lambda \int_{\Omega} u_{\varepsilon}\left(v_{\varepsilon}-u_{\varepsilon}\right) d x+2 \lambda \int_{\Omega} \frac{\partial u_{\varepsilon}}{\partial x_{k}}\left(v_{\varepsilon}-u_{\varepsilon}\right) x_{k} d x-2\mathcal{M}(W\chi_w)\int_{\Omega} v_{\varepsilon} ^2d x.
\end{aligned}\end{equation}
It follows from $(2.6)$ that $-\mathcal{M}(W\chi_w)>0$, then
\begin{equation}\begin{aligned}
2 \lambda &\leq C \int_{\partial \Omega}\left|\nabla v_{\varepsilon}\right|^{2} d \sigma+C \lambda\left\|u_{\varepsilon}-v_{\varepsilon}\right\|_{L^{2}(\Omega)}+2 \lambda\left|\int_{\Omega} \frac{\partial u_{\varepsilon}}{\partial x_{k}}\left(u_{\varepsilon}-v_{\varepsilon}\right) x_{k} d x\right|\\
&\leq C \int_{\partial \Omega}\left|\nabla v_{\varepsilon}\right|^{2} d \sigma+C \varepsilon\lambda^2+2 \lambda\left|\int_{\Omega} \frac{\partial u_{\varepsilon}}{\partial x_{k}}\left(u_{\varepsilon}-v_{\varepsilon}\right) x_{k} d x\right|.
\end{aligned}\end{equation}
where we have used $(3.14)$ to show that $||u_\varepsilon-v_\varepsilon||_{L^2(\Omega)}\leq C\varepsilon \lambda||u_\varepsilon||_{L^2(\Omega)}$.
Also, the last term on the RHS of $(6.9)$ is dominated by
\begin{equation}\begin{aligned}
&2 \lambda\left|\int_{\Omega} \frac{\partial u_{\varepsilon}}{\partial x_{k}}\left[u_{\varepsilon}-v_{\varepsilon}-\left(\Phi_{\varepsilon, j}-x_{j}\right) \frac{\partial v_{\varepsilon}}{\partial x_{j}}-\varepsilon\chi_w v_\varepsilon\right] x_{k} d x\right|+C \lambda \varepsilon\left\|\nabla u_{\varepsilon}\right\|_{L^{2}(\Omega)}\left\|\nabla v_{\varepsilon}\right\|_{L^{2}(\Omega)} & \\
\leq & C \lambda\left\|\frac{\partial u_{\varepsilon}}{\partial x_{k}} x_{k}\right\|_{H^{-1}(\Omega)}\left\|u_{\varepsilon}-v_{\varepsilon}-\left(\Phi_{\varepsilon, j}-x_{j}\right) \frac{\partial v_{\varepsilon}}{\partial x_{j}}-\varepsilon\chi_w v_\varepsilon\right\|_{H_{0}^{1}(\Omega)}+C \varepsilon \lambda^{2} \\
\leq & C \varepsilon \lambda^{2},
\end{aligned}\end{equation}
where we have used$||\nabla v_\varepsilon||_{L^2(\Omega)}\leq C\sqrt{\lambda}$ and $||\nabla^2 v_\varepsilon||_{L^2(\Omega)}\leq C\lambda$ as well as $(3.14)$. This completes the proof of $(6.6)$.\end{proof}

Let $\psi:\mathbb{R}^{d-1}\mapsto \mathbb{R}$ be a $C^2$ function with $\psi(0)=|\nabla \psi(0)|=0$. Define
\begin{equation}\begin{aligned}
Z_{r} &=Z(\psi, r)=\left\{x=\left(x^{\prime}, x_{d}\right) \in \mathbb{R}^{d}:\left|x^{\prime}\right|<r \text { and } \psi\left(x^{\prime}\right)<x_{d}<r+\psi\left(x^{\prime}\right)\right\} \\
I_{r} &=I(\psi, r)=\left\{x=\left(x^{\prime}, x_{d}\right) \in \mathbb{R}^{d}:\left|x^{\prime}\right|<r \text { and } x_{d}=\psi\left(x^{\prime}\right)\right\}.
\end{aligned}\end{equation}

\begin{lemma}Let $u\in H^1(Z_2)$. Suppose that $\operatorname{div}(A\nabla u)+Bu=0$ in $Z_2$ and $u=0$ in $I_2$ for some $B\in L^\infty(Z_2)$. Also assume that $||B||_\infty+||\nabla A||_\infty+||\nabla^2 \psi||_\infty\leq C_0$ and
\begin{equation}\int_{Z_{1}}|\nabla u|^{2} d x \geq c_{0} \int_{Z_{2}}|\nabla u|^{2} d x\end{equation}
for some $C_0>0$, $c_0>0$. Then
\begin{equation}\int_{I_{1}}|\nabla u|^{2} d \sigma \geq c \int_{Z_{2}}|\nabla u|^{2} d x,\end{equation}
where $c>0$ depends only on the ellipticity constant of $A$, $c_0$, and $C_0$.
\end{lemma}
\begin{proof}
The lemma is proved by a compactness argument and the similar result can be found in \cite[Lemma 5.3]{kenig2013estimates}, where the only difference is that the function $B$ is
replaced by some constant $E\in \mathbb{R}$. Note that for $B\in L^\infty$, the boundary $C^{1,\alpha}$ estimate as well as the unique continuation
property of solution of second-order elliptic equations with Lipschitz continuous coefficients continue to hold (e.g. see \cite{kurata1993unique}). The only difference is that we need to
replace $E_k\rightarrow E$ in $\mathbb{R}$ by $B_k\rightharpoonup B$ weak-* in $L^\infty(Z_2)$ in the compactness argument. We omit the proof and refer readers to \cite[Lemma 5.3]{kenig2013estimates} for the details.
\end{proof}

\begin{rmk}
Suppose that $\mathcal{L}_\varepsilon u_\varepsilon= \lambda u_\varepsilon$ in $Z(\psi,2\varepsilon)$ and $u_\varepsilon=0$ in $I(\psi,2\varepsilon)$ for some $\lambda>1$. Assume that $\varepsilon^{2} \lambda+\varepsilon||W||_{\infty}+\|\nabla A\|_{\infty}+\left\|\nabla^{2} \psi\right\|_{\infty} \leq C_{0}$ and
\begin{equation}\int_{Z(\psi, \varepsilon)}\left|\nabla u_{\varepsilon}\right|^{2} d x \geq c_{0} \int_{Z(\psi, 2 \varepsilon)}\left|\nabla u_{\varepsilon}\right|^{2} d x,\end{equation}for some $c_0$, $C_0>0$. Then
\begin{equation}\int_{I(\psi, \varepsilon)}\left|\nabla u_{\varepsilon}\right|^{2} d \sigma \geq \frac{c}{\varepsilon} \int_{Z(\psi, 2 \varepsilon)}\left|\nabla u_{\varepsilon}\right|^{2} d x,\end{equation}
where $c>0$ depends only on the ellipticity constant of $A$, $c_0$ and $C_0$. This is a simple consequence of Lemma 6.3. Indeed, let $w(x)=u_\varepsilon(\varepsilon x)$ and $\psi_\varepsilon(x')=\varepsilon^{-1}\psi(\varepsilon x')$. Then
$-\operatorname{div}(A\nabla w)+\varepsilon Ww=\varepsilon^2\lambda w$ in $Z(\psi_\varepsilon,2)$ and
\begin{equation}\int_{Z(\psi_\varepsilon, 1)}\left|\nabla w\right|^{2} d x \geq c_{0} \int_{Z(\psi_\varepsilon, 2 )}\left|\nabla u_{\varepsilon}\right|^{2} d x.\end{equation}
Since $\varepsilon^{2} \lambda+\varepsilon||W||_{\infty}+\|\nabla A\|_{\infty}+\left\|\nabla^{2} \psi\right\|_{\infty} \leq C_{0}$, it follows from Lemma 6.3 that
\begin{equation}\int_{I(\psi_\varepsilon, 1)}\left|\nabla u_{\varepsilon}\right|^{2} d \sigma \geq \frac{c}{\varepsilon} \int_{Z(\psi_\varepsilon, 2 )}\left|\nabla w\right|^{2} d x,\end{equation} which gives $(6.15)$.
\end{rmk}
With the help of Remark 6.4, we are ready to prove Theorem 1.6. Note that the similar proof may be found in \cite{kenig2013estimates}, and we provide it for completeness.\\

\textbf{Proof of Theorem 1.6.}
For each $P\in \partial \Omega$, after translation and rotation, we may assume that $P=(0,0)$ and
\begin{equation*}\Omega \cap B\left(P, r_{0}\right)=\left\{\left(x^{\prime}, x_{d}\right) \in \mathbb{R}^{d}: x_{d}>\psi\left(x^{\prime}\right)\right\} \cap B\left(P, r_{0}\right),\end{equation*}
where $\psi(0)=|\nabla \psi(0)|=0$ and $|\nabla^2 \psi|\leq M$. For $0<r<cr_0$ with $r_0=\text{diam}(\Omega)$, let $(\triangle(P,r),D(P,r))$ denote the pair obtained from $(I(\psi,r)Z(\psi,r))$ by this change of the coordinate system.
If $0<\varepsilon <cr_0$, we may construct a finite sequence of pairs $(\triangle(P_i,\varepsilon),D(P_i,\varepsilon))$ satisfying the following conditions
\begin{equation*}
\partial \Omega=\bigcup_{i} \Delta\left(P_{i}, \varepsilon\right) \end{equation*}
and \begin{equation}
\sum_{i} \chi_{D\left(P_{i}, 2 \varepsilon\right)} \leq C \quad \text { and } \quad \Omega_{c \varepsilon} \subset \bigcup_{i} D\left(P_{i}, \varepsilon\right).
\end{equation}Denote $\triangle_i(r)=\triangle(P_i,r)$ and $D_i(r)=D_i(P,r)$.
Suppose now that $u_\varepsilon\in H^1_0(\Omega)$, $\mathcal{L}_\varepsilon(u_\varepsilon)=\lambda u_\varepsilon$ in $\Omega$ and $||u_\varepsilon||_{L^2(\Omega)}=1$. Assume that $\lambda>1$ and $\varepsilon\lambda\leq \delta$, where $\delta=\delta(A,W,\Omega)>0$ is sufficiently small. It follows from Lemma 6.2, $(5.22)$, $(5.23)$  and $(5.24)$ in the proof of Theorem 5.3 that

\begin{equation}c \lambda \leq \frac{1}{\varepsilon} \int_{\Omega_{c \varepsilon}}\left|\nabla u_{\varepsilon}\right|^{2} d x \leq \frac{1}{\varepsilon} \int_{\Omega_{2 \varepsilon}}\left|\nabla u_{\varepsilon}\right|^{2} d x \leq C \lambda.
\end{equation}
To obtain the lower bounds for $\int_{\partial \Omega}|\nabla u_\varepsilon|^2d\sigma$, we divide $\{D_i(\varepsilon)\}$ into two groups. Say $i\in J$
if \begin{equation}\int_{D_{i}(2 \varepsilon)}\left|\nabla u_{\varepsilon}\right|^{2} d x \leq N \int_{D_{i}(\varepsilon)}\left|\nabla u_{\varepsilon}\right|^{2} d x\end{equation}
with a large constant $N=N(A,W,\Omega)$ to be determined later. Note that if $i\in J$, then by Remark 6.4,
\begin{equation}\int_{\Delta_{i}(\varepsilon)}\left|\nabla u_{\varepsilon}\right|^{2} d \sigma \geq \frac{C_1}{\varepsilon} \int_{D_{i}(\varepsilon)}\left|\nabla u_{\varepsilon}\right|^{2} d x,\end{equation}
where $C_1>0$ depends only on $A$, $W$ $\Omega$ and $N$. Consequently,
\begin{equation}\begin{aligned}
\int_{\partial \Omega}\left|\nabla u_{\varepsilon}\right|^{2} d \sigma & \geq \frac{c C_1}{\varepsilon} \int_{U_{i \in J} D_{i}(\varepsilon)}\left|\nabla u_{\varepsilon}\right|^{2} d x \\
& \geq \frac{c C_1}{\varepsilon}\left\{\int_{\Omega_{c \varepsilon}}\left|\nabla u_{\varepsilon}\right|^{2} d x-\int_{U_{i \notin J} D_{i}(\varepsilon)}\left|\nabla u_{\varepsilon}\right|^{2} d x\right\} \\
& \geq \frac{c C_1}{\varepsilon}\left\{c \varepsilon \lambda-\int_{U_{i \notin J} D_{i}(\varepsilon)}\left|\nabla u_{\varepsilon}\right|^{2} d x\right\},
\end{aligned}\end{equation}
where we have used the fact $\Omega_{c\varepsilon}\subset \cup_i D_i(\varepsilon)$ and the estimate $(6.19)$.
Finally, in view of the definition of $J$,
\begin{equation}\int_{\cup_{i \notin J} D_{i}(\varepsilon)}\left|\nabla u_{\varepsilon}\right|^{2} d x \leq \frac{C}{N} \int_{\cup_{i \notin J} D_{i}(2 \varepsilon)}\left|\nabla u_{\varepsilon}\right|^{2} d x \leq \frac{C}{N} \int_{\Omega_{2 \varepsilon}}\left|\nabla u_{\varepsilon}\right|^{2} d x \leq \frac{C \varepsilon \lambda}{N},\end{equation}
where we have used the fact $\cup_i D_i(2\varepsilon)\subset \Omega_{2\varepsilon}$. Thus, in view of $(6.22)$, we have
\begin{equation}\int_{\partial \Omega}\left|\nabla u_{\varepsilon}\right|^{2} d \sigma \geq c \gamma \lambda\left\{c-C N^{-1}\right\} \geq c \lambda,\end{equation}
if $N=N(A,W,\Omega)$ is sufficient large. Then the proof is complete.

\normalem\bibliography{rapidly_oscillating}
\bibliographystyle{plain}

\end{document}